\documentclass{article}
\usepackage[english]{babel}
\usepackage{amsmath}
\usepackage{amssymb}
\usepackage{amsthm}
\usepackage{flexisym}
\usepackage{graphicx} 
\usepackage{float}
\usepackage{tikz}
\usepackage[all]{xy}
\usepackage[nottoc]{tocbibind}
\usepackage[utf8]{inputenc}
\newtheorem{thm}{Theorem}

\newtheorem*{mthm*}{Main Theorem}
\newtheorem{lem}{Lemma}
\newtheorem{prop}{Proposition}
\newtheorem{col}{Corollary}
\newtheorem*{col*}{Corollary 1}
\newcommand{\isEquivTo}[1]{\underset{#1}{\sim}}
\newcommand{\interior}[1]{%
  {\kern0pt#1}^{\mathrm{o}}%
}
\newcommand{\vol}{\mathrm{Vol}}
\title{Low dilatation pseudo-Anosovs on punctured surfaces and volume.}
\author{Shixuan Li}
\date{}
\begin{document}
\maketitle
\begin{abstract}
For a pseudo-Anosov homeomorphism $f$ on a closed surface of genus $g\geq 2$, for which the entropy is on the order $\frac{1}{g}$ (the lowest possible order), Farb-Leininger-Margalit showed that the volume of the mapping torus is bounded, independent of $g$. We show that the analogous result fails for a surface of fixed genus $g$ with $n$ punctures, by constructing pseudo-Anosov homeomorphism with entropy of the minimal order $\frac{\log n}{n}$, and volume tending to infinity.
\end{abstract}
\section{Introduction}

Let $l(g,n) = \min\{\log(\lambda(f)) | f : S_{g,n} \to S_{g,n}\}$ denote the logarithm of the minimal dilatation of a pseudo-Anosov $f$ on an orientable surface $S_{g,n}$ with genus $g$ and $n$ punctures, that is, the minimal topological entropy. When $n=0$, Penner showed that 
\[\frac{\log2}{12g-12}<l_{g,0}<\frac{\log 11}{g}.\] 
See \cite{penner}. These bounds have been improved since Penner's original work \cite{bound1,bound2,bound3,bound4,bound5,bound6}

To better understand where minimal dilatation pseudo-Anosov homeomorphism come from, in \cite{flm}, the authors consider the set 
\[
\Psi_L=\{f:S_{g,0}\to S_{g,0} | f \text{ is pseudo-Anosov, } \log(\lambda(f)) \leq \frac{L}{g}\}.
\]
They show that for any $L>0$ there exists finite number of hyperbolic 3-manifolds $M_1, \dots, M_n$, such that for each $f\in \Psi_L$, the mapping torus $M_f$ of $f$ is obtained by Dehn fillings on some $M_i$. See  \cite[Corollary 1.4]{flm}. As a consequence, the volume of $M_f$ is bounded by a constant depending only on $L$; see \cite[Corollary 1.5]{flm}.
See also \cite{agol2,kojima,brock}.

For punctured surfaces of a fixed genus, Tsai \cite{tsai} proved that $l_{g,n}$ has a different asymptotic behavior.
\begin{thm}[Tsai]
For any fixed $g\geq2$, for all $n\geq3$, there is a constant $c_g\geq1$ depending on $g$ such that
\[ \frac{\log n}{c_gn}< l_{g,n} < \frac{c_g\log{n}}{n}.\]
\end{thm}
See also \cite{yazdi,yazdi2,valdivia,bound5}. For fixed $g\geq 2, n\geq 0$, let
\[
\Psi_{g,L}=\{f:S_{g,n}\to S_{g,n} | f \text{ is pseudo-Anosov, } \log(\lambda(f)) \leq L\frac{\log{n}}{n}\}.
\]
We show that the analogue of the results of \cite{flm} fail for $\Psi_{g,L}$. Specifically, we prove the following. 
\begin{mthm*}
For any fixed $g\geq 2$, and $L\geq 162g$, there exists a sequence $\{M_{f_i}\}_{i=1}^{\infty}$, with $f_i\in \Psi_{g,L}$, so that $\displaystyle {\lim_{n \to \infty} \vol(M_{f_i})\rightarrow \infty}$.
\end{mthm*}
As a consequence, we have the following.
\begin{col}
For any $g\geq 2$, there exists $L$ such that there is no finite set $\Omega$ of 3-manifolds so that for all $M_f$, $f\in \Psi_{g,L}$ are obtained by Dehn filling on some $M \in \Omega$.
\end{col}
The construction in the proof of the Main Theorem is based on the example in \cite{tsai} of $f_{g,n}:S_{g,n}\to S_{g,n}$ with 
\[
\log(\lambda(f_{g,n})) < \frac{c_g\log{n}}{n}.
\]
But for each $g$, one can show that $\{M_{f_{g,n}}\}_{n=1}^\infty$ are all obtained by Dehn fillings on a finite number of 3-manifolds, so we have to modify this construction. See also examples constructed by Kin-Takasawa \cite{bound5}. The idea is to compose $f_{g,n}$ with homeomorphisms supported in a subsurface of $S_{g,n}$ that become more and more complicated as $n$ gets larger. This has to be balanced with keeping the stretch factor bounded by a fixed multiple of $\frac{\log n }{n}$.

In Section 2 we recall some of the background we will need on fibered 3-manifold, hyperbolic geometry and Dehn surgery. In Section 3 we state Theorem \ref{main}, which is a version of the Main Theorem for punctured spheres based on a construction of \cite{hironaka}, then prove the Main Theorem based on that. In Section 4 we give the complete proof of Theorem \ref{main} by giving the construction of the sequence $\{M_{f_i}\}_{i=1}^{\infty}$, which are obtained by cutting open and gluing in an increasing numbers of copies of a certain manifold with totally geodesic boundary, then applying Dehn fillings.

Based on the Main Theorem, we have the following question. If we only consider the minimizers of the entropy, can we still find a sequence with unbounded volume?

\section{Background}
\subsection{Fibered 3-manifolds}
Let $S$ be a closed surface minus a finite number of points. We sometimes consider $S$ as a compact surface with boundary components, and will confuse punctures with boundary components when convenient (the former obtained from the latter by removing the boundary). The following theorem is from \cite{thurston}.

\begin{thm}[Thurston]
Any diffeomorphism $f$ on $S$ is isotopic to a map $f\textprime$ satisfying one of the following conditions:
\begin{enumerate}
\item[(i)]$f\textprime$ has finite order.
\item[(ii)]$f\textprime$ preserves a disjoint union of essential simple curves.
\item[(iii)]There exists $\lambda>1$ and two transverse measured foliations $\mathcal{F}^s$ and $\mathcal{F}^u$, called the stable and unstable foliations, respectively, such that 
\[f\textprime(\mathcal{F}^s)=(1/\lambda)\mathcal{F}^s, f\textprime(\mathcal{F}^u)=\lambda\mathcal{F}^u.\]
\end{enumerate}
\end{thm}
The three cases are called {\em periodic}, {\em reducible} and {\em pseudo-Anosov} respectively. The number $\lambda=\lambda(f)$ in case (iii) is called the {\em stretch factor} of $f$.
The topological entropy of pseudo-Anosov homeomorphism $f:S\to S$ is $\log (\lambda(f))$

Let $M$ be the interior of a compact, connected, orientable, irreducible, atoroidal 3-manifold that fibers over $S^1$ with fiber $S\subset M$ and monodromy $f$. That is, $M$ is the mapping torus of $f$: 
\[M=M_f=S\times[0,1]  / (x,1)\sim(f(x),0).\] 
Then $S$ is a closed orientable surface with a finite number of punctures and negative Euler characteristic, and $f$ is pseudo-Anosov with a unique expanding invariant foliation up to isotopy. Associated to $(M,S)$ we also have 
\begin{enumerate}
\item[(i)]$F\subset H^1(M,\mathbb{R})$, the open face of the unit ball in Thurston norm with $[S]\in (F\cdot \mathbb{R}^+)$. See \cite{thurstonnorm}.
\item[(ii)]A suspension flow $\psi$ on $M$, and a 2-dimensional foliation obtained by suspending the stable and unstable foliation of $f$. See \cite{fried1}.
\end{enumerate}
$F$ is called a {\em fibred face} of the Thurston norm ball. The segments 
\[{x}\times [0,1] \subset S \times [0,1]\]
glued together in $M_f$ are leaves of the 1-dimensional foliation $\Psi$ of M, the flow lines of $\psi$.

The following theorem is from \cite{fried1} and \cite{fried2}.
\begin{thm}[Fried]\label{fried}
Let $(M,S)$, $F$ and $\Psi$ be as above. Then any integral class in $F\cdot \mathbb{R}^+$ is represented by a fiber $S\textprime$ of a fibration of $M$ over the circle which can be isotoped to be transverse to $\Psi$, and the first return map of $\psi$ coincides with the pseudo-Anosov monodromy $f\textprime$, up to isotopy. Moreover, if $S\textprime \subset M$ is any orientable  surface with $S\textprime \pitchfork \Psi$, then $[S\textprime]\in \overline{F\cdot \mathbb{R}^+}$.
\end{thm}

If $f: S\to S$ is pseudo-Anosov on a surface with punctures, and $G\subset S$ is a spine, then we can homotope $f$ to a map $g: S\to G$ so that $g|_G:G\to G$ a graph map; that is, $g$ sends vertices to vertices and edges to edge paths. The growth rate of $g|_G$ is the largest absolute value of any eigenvalue of the Perron-Frobenious block of the transition matrix $T$ induced by $g$, and is an upper bound for $\lambda(f)$, see \cite{bestvina}.

The Perron-Frobenius Theorem tells that largest eigenvalue of a Perron-Frobenius matrix is bounded above by the largest row sum of the matrix. Recall that associated to a non-negative integral matrix $T=\{e_{ij}\}, 1\leq i,j \leq n$ is a directed graph $\Gamma$, where $\{V_1, V_2, \dots, V_n\}$ is the vertex set of $\Gamma$ corresponding to the columns/rows of $T$, and $e_{ij}$ represents the number of  edges pointing from $V_i$ to $V_j$. We have the following proposition. See \cite{gantmacher}.

\begin{prop}\label{pf}
Let $\Gamma$ be the directed graph of an integral Perron-Frobenius matrix $T$ with eigenvalue $\lambda$. Let $N(V_i,l)$ be the number of length-$l$ paths emanating from vertex $V_i$ in $\Gamma$. Then 
$\lambda \leq \max_i{{N(V_i,l)}}$.
\end{prop}

\subsection{Hyperbolic geometry}

\begin{figure}[!ht]\centering \label{AA}

\tikzset{every picture/.style={line width=0.75pt}} 



\caption{Left: $\Sigma_4$. Middle: $A_0$. Right: $A$.}  
\end{figure}

\begin{figure}[!ht]\centering \label{AA}

\tikzset{every picture/.style={line width=0.75pt}} 

\tikzset{every picture/.style={line width=0.75pt}} 

%

\caption{Left: $\Sigma_4$. Middle: $A_0$. Right: $A$.}  
\end{figure}

The following construction is given by Agol in \cite{agol}.
Let $\Sigma_4$ denote the 4-puntured sphere, and let $\delta_0, \delta_1 \subset \Sigma_4$ be two circles on $\Sigma_4$ shown in Figure 1. Let $A_0$ be $\Sigma_4\times [0,1]\backslash (\delta_0 \times \{0\} \cup \delta_1 \times \{1\})$. Let $V_8$ denote the volume of a regular, ideal, hyperbolic octahedron.

\begin{prop}[Agol]\label{agol}
$A_0$ has complete hyperbolic metric with totally geodesic boundary, with $\vol(A_0)=2V_8$.
\end{prop}
For our purpose, it is more useful to draw the 4-punctured sphere as a 3-punctured disk, then $A$ and $A_0$ are manifolds shown in Figure 2.
Let $A$ denote the manifold obtained by isometrically gluing two copies of $A_0$ along $\Sigma_4\times \{0\}\backslash (\delta_0 \times \{0\})$, then we have 
\[A\cong\Sigma_4\times [0,1]\backslash (\delta_1 \times \{0,1\} \cup \delta_0 \times \{1/2\})\]
and $A$ is a hyperbolic 3-manifold with totally geodesic boundary and \[\vol(A)=4V_8\].

We will also need the following theorem, due to Adams \cite{adams}.
\begin{thm}[Adams]\label{adams}
Any properly embedded incompressible thrice-punctured sphere in a hyperbolic 3-manifold $M$ is isotopic to a totally geodesic properly embedded thrice-punctured sphere in $M$.
\end{thm}
From this theorem one easily obtains the following.
\begin{col}
A disjoint union of pairwise non-isotopic properly embedded thrice-punctured spheres in a hyperbolic 3-manifold $M$ are simultaneously isotopic to pairwise disjoint totally geodesic thrice-punctured spheres in $M$.
\end{col}

\subsection{Dehn surgery}

Let $M$ be a compact 3-manifold with boundary $\partial M=\partial_1M\sqcup \dots \sqcup \partial_kM$ so that the interior of $M$ is a complete hyperbolic manifold, where $\partial_iM$ is a torus for any $1\leq i\leq k$. 
Choose a basis $\mu_i,\nu_i$ for $H_1(\partial_i M)=\pi_1(\partial_i M)$. Then the isotopy class of any essential simple closed curve $\beta_i$ on $\partial_iM$, called a {\em slope}, is represented by $p_i\mu_i+q_i\nu_i$ in $H_1(\partial_iM)$ for coprime integer $p_i,q_i$. Since we do not care about orientation of $\beta_i$, we use the notation $\beta_i=\frac{p_i}{q_i}\in\mathbb{Q}\cup\{\infty\}$.
Given $\beta=(\beta_1,\dots, \beta_k)$, where each $\beta_i$ is a slope, let $M_\beta$ denote the manifold obtained by gluing a solid torus to each $\partial_iM$, where $\beta_i$ is the slope $\partial_iM$ identified with the meridian of the corresponding solid torus.
We call $\beta=\{\beta_1,\dots,\beta_k\}$ the Dehn surgery coefficients.

The following is from \cite{thurstonnote,dehnsurgery}.
\begin{thm}[Thurston]\label{thurston}
If the interior of $M$ is a complete hyperbolic 3-manifold of finite volume and $\beta=\{\beta_1,\dots,\beta_k\}$ are the Dehn surgery coefficients, then for all but finitely many slopes $\beta_i$, for each $i$, $M_\beta$ is hyperbolic and $\vol(M_\beta)<\vol(M)$. If $\beta^n=(\beta^n_1,\dots, \beta^n_k)$, with $\{\beta_i^n\}_{n=1}^{\infty}$ an infinite sequence of distinct slopes on $\partial_iM$ for each $1\leq i\leq n$, then $\displaystyle{\lim_{n \to \infty} \vol(M_{\beta^n})\rightarrow \vol(M)}$.
\end{thm}

Let $||[M]||$ denote the {\em Gromov norm} of the fundamental class $[M]\in H_3(M;\partial M)$. Then we have the following two theorems. See \cite[Theorem 6.2, Proposition 6.5.2, Lemma 6.5.4]{thurstonnote}
\begin{thm}[Gromov]\label{gromov}
If the interior of $M$ admits a complete hyperbolic metric of finite volume, then
\[
||[M]||=\frac{\vol(M)}{v_3}.
\]
\end{thm}
\begin{thm}[Thurston]\label{thurston1}
For any Dehn fillings with Dehn surgery coefficients $\beta=\{\beta_1,\dots,\beta_k\}$,
\[
||[M_\beta]||\leq||[M]||.
\]
\end{thm}
We will be interested in a special case of Dehn surgery in which $M$ is obtained from a mapping torus  
\[M_f=S\times [0,1]/(x,1)\sim(f(x),0)\]
by removing neighborhoods of disjoint curves $\alpha_1, \alpha_2, \dots, \alpha_k$, $\alpha_i \subset S\times \{t_i\}$, for some \[0<t_1<t_2<\dots<t_k<1.\] 
Then we can choose basis $\mu_i,\nu_i$ of $H_1(\partial_i M)$, so that if $\beta_i=\frac{1}{r_i}$, then 
\[M_{\beta}=M_{T_{\alpha_k}^{r_k}T_{\alpha_{k-1}}^{r_{k-1}} \dots T_{\alpha_1}^{r_1}f}.\]
See, for example, \cite{stallings}.

\section{Reduction}
Consider the sphere with $n+m+2$ puntures $S_{0,n+m+2}$. We can distribute the punctures as shown in Figure 3. Let $x$, $y$ and $z$ be the three of the punctures as shown. Let $X,Y\subset S_{0,n+m+2}$ be two embedded punctured disks centered at $x$ and $y$ as shown in Figure 3. There are $n$ punctures in $X$ arranged around $x$, $m$ punctures in $Y$ arranged around $y$, with one puncture shared in $X$ and $Y$. Let $p_n$ denote the homeomorphism which is supported inside $X$, fixes $x$ and rotates the punctures around $x$ by one counterclockwise. Let $q_m$ denote the homeomorphism which is supported inside $Y$,  fixes $y$ and rotates the punctures around $y$ by one clockwise. For any $n,m>6$, let $f_{n,m}:S_{0,n+m+2} \rightarrow S_{0,n+m+2}$ be $f_{n,m}=q_mp_n$. These homeomorphisms $f_{n,m}$ were constructed by Hironaka and Kin in \cite{hironaka} and were shown to be pseudo-Anosov.

Let $V_1, V_2, \dots, V_n$ be the punctures in $X$, starting with $V_1$ in $X\cap Y$, ordered counter-clockwise, as shown in Figure 3. Let $\Sigma_0 \subset S_{0,n+m+2}$ be the subsurface containing 3 consecutive punctures $\{V_i, V_{i+1}, V_{i+2}\}$, with $\partial\Sigma_0=\beta$ as shown in Figure 3. Let $\alpha,\gamma\subset\Sigma_0$ be the two essential closed curves shown.

We will consider the composition $hf^3_{n,m}$, where $h: S_{0,n+m+2}\to  S_{0,n+m+2}$ is a homeomorphism supported in $\Sigma_0$. Note that if we replace $h$ by $p_n^khp^{-k}_n$ for $1\leq k \leq n-(i+3)$, which is supported on $p^k_n(\Sigma_0)$, then $q_m$ commutes with $p_n^jhp_n^{-j}$ for $1\leq j \leq k$. So we have

\begin{equation}\notag
\begin{split}
f^k_{n,m}hf^3_{n,m}f^{-k}_{n,m}
 & = f^{k-1}_{n,m}q_m(p_nhp_n^{-1})p_nf^{-k+3}_{n,m} \\
 & = f^{k-1}_{n,m}(p_nhp_n^{-1})q_mp_nf^{-k+3}_{n,m} \\
 & = f^{k-1}_{n,m}(p_nhp_n^{-1})f^{-k+4}_{n,m} \\
 & = f^{k-2}_{n,m}q_m(p^2_nhp_n^{-2})p_nf^{-k+4}_{n,m} \\
 & = \dots \\
 & = q_m(p_n^{k}hp_n^{-k})p_nf^{2}_{n,m} \\
 & =(p_n^{k}hp_n^{-k})f^{3}_{n,m} \\
\end{split}
\end{equation}

That is, $hf^3_{n,m}$ is conjugate to $p_n^khp^{-k}_nf^3_{n,m}$. In particular, we can assume $\Sigma_0$ surrounds $V_i,V_{i+1},V_{i+2}$ for any $2\leq i\leq n-5$ at the expense of conjugation which does not affect stretch factor or the homeomorphism type of mapping torus. For this reason, in the following statements, $\Sigma_0$ is allowed to surround the punctures $V_i,V_{i+1},V_{i+2}$ for any  $2\leq i\leq n-5$.

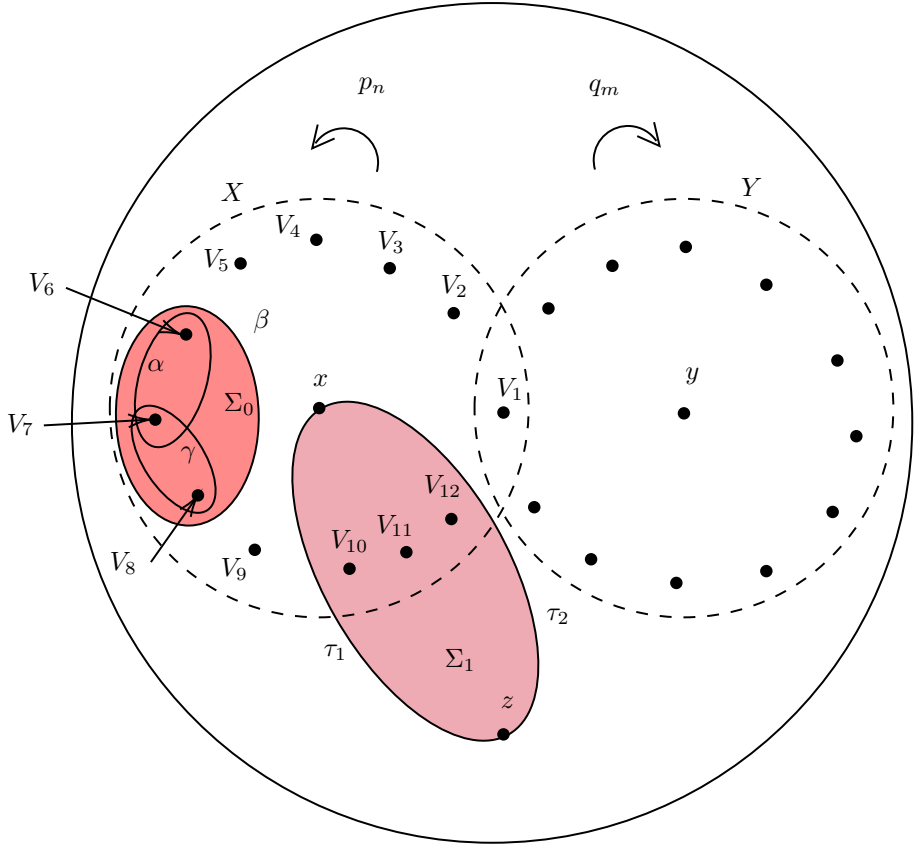
\begin{figure}[!ht]\centering

\tikzset{every picture/.style={line width=0.75pt}} 

\begin{tikzpicture}[x=0.75pt,y=0.75pt,yscale=-1,xscale=1]

\draw  [fill={rgb, 255:red, 255; green, 0; blue, 0 }  ,fill opacity=0.46 ] (134.36,270.73) .. controls (114.48,271.15) and (97.85,246.7) .. (97.21,216.12) .. controls (96.57,185.55) and (112.16,160.42) .. (132.04,160.01) .. controls (151.92,159.59) and (168.56,184.04) .. (169.2,214.62) .. controls (169.84,245.19) and (154.24,270.32) .. (134.36,270.73) -- cycle ;
\draw  [fill={rgb, 255:red, 208; green, 2; blue, 27 }  ,fill opacity=0.33 ] (201.13,211.42) .. controls (223.45,198.66) and (262.61,225.2) .. (288.61,270.68) .. controls (314.61,316.17) and (317.59,363.38) .. (295.28,376.14) .. controls (272.96,388.89) and (233.8,362.36) .. (207.8,316.87) .. controls (181.8,271.38) and (178.82,224.17) .. (201.13,211.42) -- cycle ;
\draw   (75,218.92) .. controls (75,101.88) and (169.88,7) .. (286.92,7) .. controls (403.96,7) and (498.83,101.88) .. (498.83,218.92) .. controls (498.83,335.96) and (403.96,430.83) .. (286.92,430.83) .. controls (169.88,430.83) and (75,335.96) .. (75,218.92) -- cycle ;
\draw  [dash pattern={on 4.5pt off 4.5pt}] (94.11,211.45) .. controls (94.11,153.12) and (141.39,105.84) .. (199.72,105.84) .. controls (258.05,105.84) and (305.33,153.12) .. (305.33,211.45) .. controls (305.33,269.78) and (258.05,317.06) .. (199.72,317.06) .. controls (141.39,317.06) and (94.11,269.78) .. (94.11,211.45) -- cycle ;
\draw  [dash pattern={on 4.5pt off 4.5pt}] (278.06,211.45) .. controls (278.06,153.12) and (325.34,105.84) .. (383.67,105.84) .. controls (441.99,105.84) and (489.28,153.12) .. (489.28,211.45) .. controls (489.28,269.78) and (441.99,317.06) .. (383.67,317.06) .. controls (325.34,317.06) and (278.06,269.78) .. (278.06,211.45) -- cycle ;
\draw  [fill={rgb, 255:red, 0; green, 0; blue, 0 }  ,fill opacity=1 ] (290,213.69) .. controls (290,212.23) and (291.18,211.05) .. (292.64,211.05) .. controls (294.1,211.05) and (295.28,212.23) .. (295.28,213.69) .. controls (295.28,215.15) and (294.1,216.33) .. (292.64,216.33) .. controls (291.18,216.33) and (290,215.15) .. (290,213.69) -- cycle ;
\draw  [fill={rgb, 255:red, 0; green, 0; blue, 0 }  ,fill opacity=1 ] (290,376.14) .. controls (290,374.68) and (291.18,373.5) .. (292.64,373.5) .. controls (294.1,373.5) and (295.28,374.68) .. (295.28,376.14) .. controls (295.28,377.59) and (294.1,378.77) .. (292.64,378.77) .. controls (291.18,378.77) and (290,377.59) .. (290,376.14) -- cycle ;
\draw  [fill={rgb, 255:red, 0; green, 0; blue, 0 }  ,fill opacity=1 ] (197.08,211.45) .. controls (197.08,209.99) and (198.26,208.81) .. (199.72,208.81) .. controls (201.18,208.81) and (202.36,209.99) .. (202.36,211.45) .. controls (202.36,212.91) and (201.18,214.09) .. (199.72,214.09) .. controls (198.26,214.09) and (197.08,212.91) .. (197.08,211.45) -- cycle ;
\draw  [fill={rgb, 255:red, 0; green, 0; blue, 0 }  ,fill opacity=1 ] (381.03,214.09) .. controls (381.03,212.63) and (382.21,211.45) .. (383.67,211.45) .. controls (385.12,211.45) and (386.31,212.63) .. (386.31,214.09) .. controls (386.31,215.55) and (385.12,216.73) .. (383.67,216.73) .. controls (382.21,216.73) and (381.03,215.55) .. (381.03,214.09) -- cycle ;
\draw  [fill={rgb, 255:red, 0; green, 0; blue, 0 }  ,fill opacity=1 ] (264.92,163.52) .. controls (264.92,162.07) and (266.1,160.89) .. (267.56,160.89) .. controls (269.01,160.89) and (270.19,162.07) .. (270.19,163.52) .. controls (270.19,164.98) and (269.01,166.16) .. (267.56,166.16) .. controls (266.1,166.16) and (264.92,164.98) .. (264.92,163.52) -- cycle ;
\draw  [fill={rgb, 255:red, 0; green, 0; blue, 0 }  ,fill opacity=1 ] (232.67,140.83) .. controls (232.67,139.37) and (233.85,138.19) .. (235.31,138.19) .. controls (236.76,138.19) and (237.94,139.37) .. (237.94,140.83) .. controls (237.94,142.29) and (236.76,143.47) .. (235.31,143.47) .. controls (233.85,143.47) and (232.67,142.29) .. (232.67,140.83) -- cycle ;
\draw  [fill={rgb, 255:red, 0; green, 0; blue, 0 }  ,fill opacity=1 ] (263.72,267.44) .. controls (263.72,265.98) and (264.91,264.8) .. (266.36,264.8) .. controls (267.82,264.8) and (269,265.98) .. (269,267.44) .. controls (269,268.9) and (267.82,270.08) .. (266.36,270.08) .. controls (264.91,270.08) and (263.72,268.9) .. (263.72,267.44) -- cycle ;
\draw  [fill={rgb, 255:red, 0; green, 0; blue, 0 }  ,fill opacity=1 ] (312.7,161.13) .. controls (312.7,159.68) and (313.88,158.5) .. (315.33,158.5) .. controls (316.79,158.5) and (317.97,159.68) .. (317.97,161.13) .. controls (317.97,162.59) and (316.79,163.77) .. (315.33,163.77) .. controls (313.88,163.77) and (312.7,162.59) .. (312.7,161.13) -- cycle ;
\draw  [fill={rgb, 255:red, 0; green, 0; blue, 0 }  ,fill opacity=1 ] (344.95,139.63) .. controls (344.95,138.18) and (346.13,137) .. (347.59,137) .. controls (349.04,137) and (350.22,138.18) .. (350.22,139.63) .. controls (350.22,141.09) and (349.04,142.27) .. (347.59,142.27) .. controls (346.13,142.27) and (344.95,141.09) .. (344.95,139.63) -- cycle ;
\draw  [fill={rgb, 255:red, 0; green, 0; blue, 0 }  ,fill opacity=1 ] (381.98,130.08) .. controls (381.98,128.62) and (383.16,127.44) .. (384.61,127.44) .. controls (386.07,127.44) and (387.25,128.62) .. (387.25,130.08) .. controls (387.25,131.54) and (386.07,132.72) .. (384.61,132.72) .. controls (383.16,132.72) and (381.98,131.54) .. (381.98,130.08) -- cycle ;
\draw  [fill={rgb, 255:red, 0; green, 0; blue, 0 }  ,fill opacity=1 ] (241.03,284.16) .. controls (241.03,282.71) and (242.21,281.53) .. (243.67,281.53) .. controls (245.12,281.53) and (246.31,282.71) .. (246.31,284.16) .. controls (246.31,285.62) and (245.12,286.8) .. (243.67,286.8) .. controls (242.21,286.8) and (241.03,285.62) .. (241.03,284.16) -- cycle ;
\draw  [fill={rgb, 255:red, 0; green, 0; blue, 0 }  ,fill opacity=1 ] (195.64,126.5) .. controls (195.64,125.04) and (196.82,123.86) .. (198.28,123.86) .. controls (199.73,123.86) and (200.92,125.04) .. (200.92,126.5) .. controls (200.92,127.95) and (199.73,129.13) .. (198.28,129.13) .. controls (196.82,129.13) and (195.64,127.95) .. (195.64,126.5) -- cycle ;
\draw  [fill={rgb, 255:red, 0; green, 0; blue, 0 }  ,fill opacity=1 ] (157.42,138.44) .. controls (157.42,136.98) and (158.6,135.8) .. (160.06,135.8) .. controls (161.51,135.8) and (162.69,136.98) .. (162.69,138.44) .. controls (162.69,139.9) and (161.51,141.08) .. (160.06,141.08) .. controls (158.6,141.08) and (157.42,139.9) .. (157.42,138.44) -- cycle ;
\draw  [fill={rgb, 255:red, 0; green, 0; blue, 0 }  ,fill opacity=1 ] (212.36,292.53) .. controls (212.36,291.07) and (213.54,289.89) .. (215,289.89) .. controls (216.46,289.89) and (217.64,291.07) .. (217.64,292.53) .. controls (217.64,293.98) and (216.46,295.16) .. (215,295.16) .. controls (213.54,295.16) and (212.36,293.98) .. (212.36,292.53) -- cycle ;
\draw  [fill={rgb, 255:red, 0; green, 0; blue, 0 }  ,fill opacity=1 ] (164.58,282.97) .. controls (164.58,281.51) and (165.77,280.33) .. (167.22,280.33) .. controls (168.68,280.33) and (169.86,281.51) .. (169.86,282.97) .. controls (169.86,284.43) and (168.68,285.61) .. (167.22,285.61) .. controls (165.77,285.61) and (164.58,284.43) .. (164.58,282.97) -- cycle ;
\draw  [fill={rgb, 255:red, 0; green, 0; blue, 0 }  ,fill opacity=1 ] (135.92,255.5) .. controls (135.92,254.04) and (137.1,252.86) .. (138.56,252.86) .. controls (140.01,252.86) and (141.19,254.04) .. (141.19,255.5) .. controls (141.19,256.95) and (140.01,258.13) .. (138.56,258.13) .. controls (137.1,258.13) and (135.92,256.95) .. (135.92,255.5) -- cycle ;
\draw  [fill={rgb, 255:red, 0; green, 0; blue, 0 }  ,fill opacity=1 ] (305.53,261.47) .. controls (305.53,260.01) and (306.71,258.83) .. (308.17,258.83) .. controls (309.62,258.83) and (310.81,260.01) .. (310.81,261.47) .. controls (310.81,262.93) and (309.62,264.11) .. (308.17,264.11) .. controls (306.71,264.11) and (305.53,262.93) .. (305.53,261.47) -- cycle ;
\draw  [fill={rgb, 255:red, 0; green, 0; blue, 0 }  ,fill opacity=1 ] (334.2,287.75) .. controls (334.2,286.29) and (335.38,285.11) .. (336.84,285.11) .. controls (338.29,285.11) and (339.47,286.29) .. (339.47,287.75) .. controls (339.47,289.2) and (338.29,290.39) .. (336.84,290.39) .. controls (335.38,290.39) and (334.2,289.2) .. (334.2,287.75) -- cycle ;
\draw  [fill={rgb, 255:red, 0; green, 0; blue, 0 }  ,fill opacity=1 ] (377.2,299.69) .. controls (377.2,298.24) and (378.38,297.05) .. (379.84,297.05) .. controls (381.29,297.05) and (382.47,298.24) .. (382.47,299.69) .. controls (382.47,301.15) and (381.29,302.33) .. (379.84,302.33) .. controls (378.38,302.33) and (377.2,301.15) .. (377.2,299.69) -- cycle ;
\draw  [fill={rgb, 255:red, 0; green, 0; blue, 0 }  ,fill opacity=1 ] (422.59,293.72) .. controls (422.59,292.26) and (423.77,291.08) .. (425.22,291.08) .. controls (426.68,291.08) and (427.86,292.26) .. (427.86,293.72) .. controls (427.86,295.18) and (426.68,296.36) .. (425.22,296.36) .. controls (423.77,296.36) and (422.59,295.18) .. (422.59,293.72) -- cycle ;
\draw  [fill={rgb, 255:red, 0; green, 0; blue, 0 }  ,fill opacity=1 ] (456.03,263.86) .. controls (456.03,262.4) and (457.21,261.22) .. (458.67,261.22) .. controls (460.13,261.22) and (461.31,262.4) .. (461.31,263.86) .. controls (461.31,265.31) and (460.13,266.5) .. (458.67,266.5) .. controls (457.21,266.5) and (456.03,265.31) .. (456.03,263.86) -- cycle ;
\draw  [fill={rgb, 255:red, 0; green, 0; blue, 0 }  ,fill opacity=1 ] (467.98,225.64) .. controls (467.98,224.18) and (469.16,223) .. (470.61,223) .. controls (472.07,223) and (473.25,224.18) .. (473.25,225.64) .. controls (473.25,227.09) and (472.07,228.27) .. (470.61,228.27) .. controls (469.16,228.27) and (467.98,227.09) .. (467.98,225.64) -- cycle ;
\draw  [fill={rgb, 255:red, 0; green, 0; blue, 0 }  ,fill opacity=1 ] (114.42,217.27) .. controls (114.42,215.82) and (115.6,214.64) .. (117.05,214.64) .. controls (118.51,214.64) and (119.69,215.82) .. (119.69,217.27) .. controls (119.69,218.73) and (118.51,219.91) .. (117.05,219.91) .. controls (115.6,219.91) and (114.42,218.73) .. (114.42,217.27) -- cycle ;
\draw  [draw opacity=0] (197.81,77.81) .. controls (200.67,73.6) and (205.36,70.69) .. (210.83,70.28) .. controls (220.36,69.57) and (228.65,76.72) .. (229.36,86.24) .. controls (229.52,88.32) and (229.3,90.33) .. (228.77,92.22) -- (212.12,87.53) -- cycle ; \draw   (197.81,77.81) .. controls (200.67,73.6) and (205.36,70.69) .. (210.83,70.28) .. controls (220.36,69.57) and (228.65,76.72) .. (229.36,86.24) .. controls (229.52,88.32) and (229.3,90.33) .. (228.77,92.22) ;
\draw   (207.47,75.36) -- (196.19,80.35) -- (198.97,68.33) ;

\draw  [draw opacity=0] (369.76,76.62) .. controls (366.89,72.41) and (362.21,69.5) .. (356.74,69.09) .. controls (347.21,68.38) and (338.91,75.52) .. (338.2,85.05) .. controls (338.05,87.12) and (338.27,89.14) .. (338.8,91.03) -- (355.45,86.33) -- cycle ; \draw   (369.76,76.62) .. controls (366.89,72.41) and (362.21,69.5) .. (356.74,69.09) .. controls (347.21,68.38) and (338.91,75.52) .. (338.2,85.05) .. controls (338.05,87.12) and (338.27,89.14) .. (338.8,91.03) ;
\draw   (360.1,74.16) -- (371.37,79.15) -- (368.6,67.14) ;

\draw  [fill={rgb, 255:red, 0; green, 0; blue, 0 }  ,fill opacity=1 ] (129.95,174.27) .. controls (129.95,172.82) and (131.13,171.64) .. (132.58,171.64) .. controls (134.04,171.64) and (135.22,172.82) .. (135.22,174.27) .. controls (135.22,175.73) and (134.04,176.91) .. (132.58,176.91) .. controls (131.13,176.91) and (129.95,175.73) .. (129.95,174.27) -- cycle ;
\draw  [fill={rgb, 255:red, 0; green, 0; blue, 0 }  ,fill opacity=1 ] (422.59,149.19) .. controls (422.59,147.73) and (423.77,146.55) .. (425.22,146.55) .. controls (426.68,146.55) and (427.86,147.73) .. (427.86,149.19) .. controls (427.86,150.65) and (426.68,151.83) .. (425.22,151.83) .. controls (423.77,151.83) and (422.59,150.65) .. (422.59,149.19) -- cycle ;
\draw  [fill={rgb, 255:red, 0; green, 0; blue, 0 }  ,fill opacity=1 ] (458.42,187.41) .. controls (458.42,185.96) and (459.6,184.78) .. (461.06,184.78) .. controls (462.52,184.78) and (463.7,185.96) .. (463.7,187.41) .. controls (463.7,188.87) and (462.52,190.05) .. (461.06,190.05) .. controls (459.6,190.05) and (458.42,188.87) .. (458.42,187.41) -- cycle ;
\draw   (135.62,163.85) .. controls (144.69,166.49) and (147.69,183.63) .. (142.3,202.13) .. controls (136.92,220.63) and (125.19,233.48) .. (116.11,230.84) .. controls (107.03,228.2) and (104.04,211.06) .. (109.42,192.56) .. controls (114.81,174.06) and (126.54,161.21) .. (135.62,163.85) -- cycle ;
\draw   (109.28,211.55) .. controls (116.25,206.94) and (129.51,214.72) .. (138.9,228.92) .. controls (148.29,243.12) and (150.26,258.37) .. (143.28,262.98) .. controls (136.31,267.59) and (123.05,259.81) .. (113.66,245.61) .. controls (104.27,231.41) and (102.31,216.16) .. (109.28,211.55) -- cycle ;
\draw    (72,150.75) -- (128.09,173.52) ;
\draw [shift={(129.95,174.27)}, rotate = 202.1] [color={rgb, 255:red, 0; green, 0; blue, 0 }  ][line width=0.75]    (10.93,-3.29) .. controls (6.95,-1.4) and (3.31,-0.3) .. (0,0) .. controls (3.31,0.3) and (6.95,1.4) .. (10.93,3.29)   ;

\draw    (60.83,220.17) -- (112.42,217.38) ;
\draw [shift={(114.42,217.27)}, rotate = 536.9100000000001] [color={rgb, 255:red, 0; green, 0; blue, 0 }  ][line width=0.75]    (10.93,-3.29) .. controls (6.95,-1.4) and (3.31,-0.3) .. (0,0) .. controls (3.31,0.3) and (6.95,1.4) .. (10.93,3.29)   ;

\draw    (114.83,289.17) -- (137.4,257.13) ;
\draw [shift={(138.56,255.5)}, rotate = 485.17] [color={rgb, 255:red, 0; green, 0; blue, 0 }  ][line width=0.75]    (10.93,-3.29) .. controls (6.95,-1.4) and (3.31,-0.3) .. (0,0) .. controls (3.31,0.3) and (6.95,1.4) .. (10.93,3.29)   ;

\draw (200.42,197.42) node   {$x$};
\draw (387.95,195.03) node   {$y$};
\draw (294.78,359.86) node   {$z$};
\draw (226.7,48.81) node   {$p_{n}$};
\draw (343.75,50) node   {$q_{m}$};
\draw (270.89,338.76) node   {$\Sigma _{1}$};
\draw (156.22,102.26) node   {$X$};
\draw (417.81,99.87) node   {$Y$};
\draw (159.5,209.54) node   {$\Sigma _{0}$};
\draw (207.92,334.55) node   {$\tau _{1}$};
\draw (320.2,316.63) node   {$\tau _{2}$};
\draw (170.62,167.51) node   {$\beta $};
\draw (117.37,189.94) node   {$\alpha $};
\draw (133.81,235.28) node   {$\gamma $};
\draw (296,201.75) node   {$V_{1}$};
\draw (268,149.75) node   {$V_{2}$};
\draw (236,127.75) node   {$V_{3}$};
\draw (184,119.75) node   {$V_{4}$};
\draw (148,136.75) node   {$V_{5}$};
\draw (60,148.75) node   {$V_{6}$};
\draw (49,219.75) node   {$V_{7}$};
\draw (101,289.75) node   {$V_{8}$};
\draw (157,292.75) node   {$V_{9}$};
\draw (215,278.75) node   {$V_{10}$};
\draw (238,270.75) node   {$V_{11}$};
\draw (262,250.75) node   {$V_{12}$};

\end{tikzpicture}

\caption{$S_{0,n+m+2}$ for $n=m=12$}  
\end{figure}

\begin{thm}\label{main}
For any $k=1,2,3,\dots$, there exists $B_k$ such that if 
\[h_k=T_{\alpha}^{u_1}T_{\gamma}^{v_1}\dots T_{\alpha}^{u_{k-1}}T_{\gamma}^{v_{k-1}}T_{\alpha}^{u_k}T_{\beta}^{v_k}\] 
where $u_i,v_i\geq B_k$ for all $i$, then for $h_kf_{n,m}: S_{0,n+m+2}\to S_{0,n+m+2}$, we have
\begin{enumerate}
\item[(1)] $h_kf^3_{n,m}$ is pseudo-Anosov.
\item[(2)] $\vol(M_{h_kf^3_{n,m}})\geq3kV_8$. 
\item[(3)] there exists $N=N_k$, such that if $n=m>N$, then 
\[
\log\lambda (h_kf^3_{n,n})\leq 54\frac{\log(2n+2)}{2n+2}.
\]
\end{enumerate}

\end{thm}
Assuming this theorem, we prove the Main Theorem from the introduction.

\begin{mthm*}
For any fixed $g\geq 2$, and $L\geq 162g$, there exists a sequence $\{M_{f_i}\}_{i=1}^{\infty}$, with $f_i\in \Psi_{g,L}$, so that $\displaystyle {\lim_{n \to \infty} \vol(M_{f_i})\rightarrow \infty}$.
\end{mthm*}
\begin{proof}
For any $g\geq 2$, \cite{tsai} gives a construction of an appropriate cover $\pi: S_{g,s}\rightarrow S_{0,n+m+2}$ such that $s=(2g+1)(n+m+1)+1$ and 
\[f_{n,m}: S_{0,n+m+2}\rightarrow S_{0,n+m+2},\] 
lifts to $S_{g,s}$. Moreover, it is clear from her construction that each of $\alpha, \beta, \gamma$ lift, so $h_k$ lifts. 

Let $\widetilde{f_k}: S_{g,s}\rightarrow S_{g,s}$ be the lift of $h_k\circ f^3_{n,m}$. Then $\log(\lambda(\widetilde{f_k}))=\log(\lambda(h_kf^3_{n,m}))$. Also by Theorem \ref{main}, for $n=m>N_k$ and $n=m$ large enough,
\[
 \log(\lambda(\widetilde{f_k})) \leq 54\frac{\log(n+m+2)}{n+m+2}<54\frac{\log(s)}{\frac{s-1}{2g+1}+1}<162g\frac{\log s}{s}.
\]
Furthermore, $\vol(M_{\widetilde{f_k}})=deg(\pi)\vol(M) \geq 3kV_8deg(\pi)$. Therefore, $\{M_{\widetilde{f_k}}\}^{\infty}_{k=1}$ is contained in the set for the theorem and $\vol(M_{\widetilde{f_k}})\rightarrow \infty$. 

\end{proof}
\begin{col*}
For any $g\geq 2$, there exists $L$ such that there is no finite set $\Omega$ of 3-manifolds so that all $M_f$, $f\in \Psi_{g,L}$, are obtained by Dehn filling on some $M \in \Omega$.
\end{col*}
\begin{proof}
Let $L\geq 162g$. If the finite set $\Omega$ exist, then by Theorems \ref{gromov} and \ref{thurston1}, 
\[
\displaystyle{\vol(M_f)\leq v_3\max_{M\in \Omega}\{||[M]||\}}<\infty,
\]
which contradicts the Main Theorem.
\end{proof}

\section{Proof of Theorem \ref{main}}
Now fix some $n,m>6$, let $f=f^3_{n,m}$. Let $M_f$ be the mapping torus.

The proof of the following lemma is almost identical to the proof of \cite[Theorem B]{long}.
\begin{lem}
$M_f\backslash ((\alpha \cup \beta) \times \{1/2\})$ is hyperbolic.
\end{lem}

\begin{proof}
Let 
\[
\Sigma=S_{0,n+m+2}\times \{1/2\}, \Sigma\textprime=\Sigma\backslash ((\alpha \cup \beta) \times \{1/2\})\subset M_f.
\]
Let $T_0\subset M_f$ be an embedded incompressible torus. By applying some isotopy, we can make every component of $T_0\backslash \Sigma\textprime$ be an annulus. Any annulus component should either miss no fiber or have boundary components parallel to $\alpha$ or $\beta$, and on opposite sides of some small neighborhood of $\alpha$ or $\beta$. Since $\alpha$ and $\beta$ bound different number of punctures, a component parallel to $\alpha$ can never connect to a component parallel to $\beta$. Also, $f^{k_1}(\alpha)$ will never close up with $f^{k_2}(\alpha)$ if $k_1\neq k_2$ since $f$ is pseudo-Anosov. By Thurston's hyperbolization theorem (see \cite{thurston1,morgan,otal}), $M_f\backslash ((\alpha \cup \beta) \times \{1/2\})$ is hyperbolic.
\end{proof}

For any $k$, let $L_k \subset M_f $ be 
\[
L_k=\alpha \times \left\{\frac{2}{4k}, \frac{4}{4k}, \dots, \frac{2k+2}{4k}\right\} \cup \gamma \times \left\{\frac{3}{4k}, \frac{5}{4k}, \dots, \frac{2k+1}{4k}\right\} \cup \beta \times \left\{\frac{1}{4k}\right\}. 
\]
Let $N(L_k)$ denote an tubular neighborhood of $L_k$ and $M_k=M_f\backslash N(L_k)$. We can order the boundary components of $M_k$ as \[\partial M_k=\partial_1M_k\sqcup \dots \sqcup \partial_{2k+2}M_k
,\] where 
\[
\begin{cases} 
   \partial_{2i}M_k=\alpha \times \{ \frac{2i}{4k}\} & \text{for any } 1\leq i \leq k+1 \\
   \partial_{2i+1}M_k=\gamma \times \{ \frac{2i+1}{4k}\}       & \text{for any } 1\leq i \leq k-1\\
   \partial_1M_k=\beta \times \{ \frac{1}{4k}\}. & 
  \end{cases}
\]

\begin{lem}
The interior of $M_f\backslash N(L_k)$ is hyperbolic and 
\[\vol(int (M_f\backslash N(L_k)))\geq 4kV_8.\]
\end{lem}
\begin{proof}
Glue $k$ copies of $A$, top to bottom, to get
 \[A_k\cong (S_{0,4}\times [0,1])\backslash \left(\alpha \times \left\{\frac{0}{2k}, \frac{2}{2k}, \dots, \frac{2k}{2k}\right\} \cup \gamma \times \left\{\frac{1}{2k}, \frac{3}{2k}, \dots, \frac{2k-1}{2k}\right\}\right),\]
with the $i$-th copy identifying with
 \[
\left(S_{0,4}\times \left[\frac{2i-2}{2k},\frac{2i}{2k}\right]\right)\backslash\left(\alpha \times \left\{\frac{2i-2}{2k}, \frac{2i}{2k}\right\}\cup \gamma \times\left\{\frac{2i-1}{2k}\right\}\right).
\]
By Theorem \ref{adams}, $A_k$ has four totally geodesic thrice-punctured sphere boundary components, and $\vol(A_k)=4kV_8$.

Cut $M_f\backslash ((\alpha \cup \beta) \times \{1/2\})$ along the two thrice-punctured spheres, i.e. the two regions shown in Figure 4. The two thrice-punctured spheres can be assumed to be totally geodesic by Corollary 2. So the cut-open manifold has four totally geodesic thrice-punctured sphere boundary components. Now glue the top boundary of $A_k$ to the top of the cut by an isometry, with the marked curves and colored faces glued correspondingly. Then apply the same to the bottom boundary. After applying an isotopy to adjust the height, we see that the result is homeomorphic to $M_f\backslash N(L_k)$. Moreover, $A_k$ is isometrically embedded in $M_f\backslash N(L_k)$. Since $\vol(A_k)\geq 4kV_8$, we have $\vol(M_f\backslash N(L_k))\geq 4kV_8$.

\begin{figure}[!ht]\centering

\tikzset{every picture/.style={line width=0.75pt}} 

\begin{tikzpicture}[x=0.75pt,y=0.75pt,yscale=-1,xscale=1]

\draw  [fill={rgb, 255:red, 249; green, 46; blue, 46 }  ,fill opacity=0.64 ] (80.45,147.49) .. controls (80.45,134.3) and (105.33,123.61) .. (136.02,123.61) .. controls (166.71,123.61) and (191.59,134.3) .. (191.59,147.49) .. controls (191.59,160.68) and (166.71,171.37) .. (136.02,171.37) .. controls (105.33,171.37) and (80.45,160.68) .. (80.45,147.49) -- cycle ;
\draw  [fill={rgb, 255:red, 126; green, 211; blue, 33 }  ,fill opacity=1 ] (93.63,147.39) .. controls (93.63,141.54) and (104.67,136.8) .. (118.29,136.8) .. controls (131.91,136.8) and (142.96,141.54) .. (142.96,147.39) .. controls (142.96,153.25) and (131.91,157.99) .. (118.29,157.99) .. controls (104.67,157.99) and (93.63,153.25) .. (93.63,147.39) -- cycle ;
\draw   (290.85,56.44) -- (290.12,211.64) .. controls (290.01,234.16) and (234.8,252.42) .. (166.79,252.42) .. controls (98.79,252.42) and (43.75,234.16) .. (43.85,211.64) -- (44.59,56.44) .. controls (44.69,33.92) and (99.91,15.67) .. (167.91,15.67) .. controls (235.91,15.67) and (290.96,33.92) .. (290.85,56.44) .. controls (290.74,78.96) and (235.53,97.21) .. (167.53,97.21) .. controls (99.52,97.21) and (44.48,78.96) .. (44.59,56.44) ;
\draw  [draw opacity=0] (290.68,150.84) .. controls (285.43,171.31) and (232.44,187.4) .. (167.85,187.4) .. controls (99.77,187.4) and (44.59,169.53) .. (44.59,147.49) .. controls (44.59,147.23) and (44.59,146.97) .. (44.61,146.72) -- (167.85,147.49) -- cycle ; \draw   (290.68,150.84) .. controls (285.43,171.31) and (232.44,187.4) .. (167.85,187.4) .. controls (99.77,187.4) and (44.59,169.53) .. (44.59,147.49) .. controls (44.59,147.23) and (44.59,146.97) .. (44.61,146.72) ;
\draw  [draw opacity=0][dash pattern={on 4.5pt off 4.5pt}] (44.61,146.72) .. controls (45.74,124.99) and (100.49,107.49) .. (167.85,107.49) .. controls (233.77,107.49) and (287.6,124.24) .. (290.95,145.32) -- (167.85,147.4) -- cycle ; \draw  [dash pattern={on 4.5pt off 4.5pt}] (44.61,146.72) .. controls (45.74,124.99) and (100.49,107.49) .. (167.85,107.49) .. controls (233.77,107.49) and (287.6,124.24) .. (290.95,145.32) ;
\draw  [fill={rgb, 255:red, 0; green, 0; blue, 0 }  ,fill opacity=1 ] (97.45,147.43) .. controls (97.45,146.2) and (98.45,145.2) .. (99.69,145.2) .. controls (100.92,145.2) and (101.92,146.2) .. (101.92,147.43) .. controls (101.92,148.66) and (100.92,149.66) .. (99.69,149.66) .. controls (98.45,149.66) and (97.45,148.66) .. (97.45,147.43) -- cycle ;
\draw  [fill={rgb, 255:red, 0; green, 0; blue, 0 }  ,fill opacity=1 ] (133.79,147.49) .. controls (133.79,146.26) and (134.79,145.26) .. (136.02,145.26) .. controls (137.25,145.26) and (138.25,146.26) .. (138.25,147.49) .. controls (138.25,148.72) and (137.25,149.72) .. (136.02,149.72) .. controls (134.79,149.72) and (133.79,148.72) .. (133.79,147.49) -- cycle ;
\draw  [fill={rgb, 255:red, 0; green, 0; blue, 0 }  ,fill opacity=1 ] (167.85,147.49) .. controls (167.85,146.26) and (168.84,145.26) .. (170.08,145.26) .. controls (171.31,145.26) and (172.31,146.26) .. (172.31,147.49) .. controls (172.31,148.72) and (171.31,149.72) .. (170.08,149.72) .. controls (168.84,149.72) and (167.85,148.72) .. (167.85,147.49) -- cycle ;
\draw   (397.82,253.79) .. controls (397.82,246.97) and (413.42,241.45) .. (432.66,241.45) .. controls (451.9,241.45) and (467.5,246.97) .. (467.5,253.79) .. controls (467.5,260.61) and (451.9,266.14) .. (432.66,266.14) .. controls (413.42,266.14) and (397.82,260.61) .. (397.82,253.79) -- cycle ;
\draw  [color={rgb, 255:red, 255; green, 255; blue, 255 }  ,draw opacity=1 ][fill={rgb, 255:red, 255; green, 255; blue, 255 }  ,fill opacity=1 ] (410.33,242.52) .. controls (410.33,237.86) and (414.11,234.08) .. (418.77,234.08) .. controls (423.43,234.08) and (427.21,237.86) .. (427.21,242.52) .. controls (427.21,247.18) and (423.43,250.96) .. (418.77,250.96) .. controls (414.11,250.96) and (410.33,247.18) .. (410.33,242.52) -- cycle ;

\draw  [color={rgb, 255:red, 255; green, 255; blue, 255 }  ,draw opacity=1 ][fill={rgb, 255:red, 255; green, 255; blue, 255 }  ,fill opacity=1 ] (438.91,242.52) .. controls (438.91,237.86) and (442.69,234.08) .. (447.36,234.08) .. controls (452.02,234.08) and (455.8,237.86) .. (455.8,242.52) .. controls (455.8,247.18) and (452.02,250.96) .. (447.36,250.96) .. controls (442.69,250.96) and (438.91,247.18) .. (438.91,242.52) -- cycle ;

\draw   (426.82,218.79) .. controls (426.82,211.97) and (442.42,206.45) .. (461.66,206.45) .. controls (480.9,206.45) and (496.5,211.97) .. (496.5,218.79) .. controls (496.5,225.61) and (480.9,231.14) .. (461.66,231.14) .. controls (442.42,231.14) and (426.82,225.61) .. (426.82,218.79) -- cycle ;
\draw  [color={rgb, 255:red, 255; green, 255; blue, 255 }  ,draw opacity=1 ][fill={rgb, 255:red, 255; green, 255; blue, 255 }  ,fill opacity=1 ] (439.33,207.52) .. controls (439.33,202.86) and (443.11,199.08) .. (447.77,199.08) .. controls (452.43,199.08) and (456.21,202.86) .. (456.21,207.52) .. controls (456.21,212.18) and (452.43,215.96) .. (447.77,215.96) .. controls (443.11,215.96) and (439.33,212.18) .. (439.33,207.52) -- cycle ;

\draw  [color={rgb, 255:red, 255; green, 255; blue, 255 }  ,draw opacity=1 ][fill={rgb, 255:red, 255; green, 255; blue, 255 }  ,fill opacity=1 ] (467.91,207.52) .. controls (467.91,202.86) and (471.69,199.08) .. (476.36,199.08) .. controls (481.02,199.08) and (484.8,202.86) .. (484.8,207.52) .. controls (484.8,212.18) and (481.02,215.96) .. (476.36,215.96) .. controls (471.69,215.96) and (467.91,212.18) .. (467.91,207.52) -- cycle ;

\draw   (397.82,184.79) .. controls (397.82,177.97) and (413.42,172.45) .. (432.66,172.45) .. controls (451.9,172.45) and (467.5,177.97) .. (467.5,184.79) .. controls (467.5,191.61) and (451.9,197.14) .. (432.66,197.14) .. controls (413.42,197.14) and (397.82,191.61) .. (397.82,184.79) -- cycle ;
\draw  [color={rgb, 255:red, 255; green, 255; blue, 255 }  ,draw opacity=1 ][fill={rgb, 255:red, 255; green, 255; blue, 255 }  ,fill opacity=1 ] (410.33,173.52) .. controls (410.33,168.86) and (414.11,165.08) .. (418.77,165.08) .. controls (423.43,165.08) and (427.21,168.86) .. (427.21,173.52) .. controls (427.21,178.18) and (423.43,181.96) .. (418.77,181.96) .. controls (414.11,181.96) and (410.33,178.18) .. (410.33,173.52) -- cycle ;

\draw  [color={rgb, 255:red, 255; green, 255; blue, 255 }  ,draw opacity=1 ][fill={rgb, 255:red, 255; green, 255; blue, 255 }  ,fill opacity=1 ] (438.91,173.52) .. controls (438.91,168.86) and (442.69,165.08) .. (447.36,165.08) .. controls (452.02,165.08) and (455.8,168.86) .. (455.8,173.52) .. controls (455.8,178.18) and (452.02,181.96) .. (447.36,181.96) .. controls (442.69,181.96) and (438.91,178.18) .. (438.91,173.52) -- cycle ;

\draw   (428.82,151.79) .. controls (428.82,144.97) and (444.42,139.45) .. (463.66,139.45) .. controls (482.9,139.45) and (498.5,144.97) .. (498.5,151.79) .. controls (498.5,158.61) and (482.9,164.14) .. (463.66,164.14) .. controls (444.42,164.14) and (428.82,158.61) .. (428.82,151.79) -- cycle ;
\draw  [color={rgb, 255:red, 255; green, 255; blue, 255 }  ,draw opacity=1 ][fill={rgb, 255:red, 255; green, 255; blue, 255 }  ,fill opacity=1 ] (441.33,140.52) .. controls (441.33,135.86) and (445.11,132.08) .. (449.77,132.08) .. controls (454.43,132.08) and (458.21,135.86) .. (458.21,140.52) .. controls (458.21,145.18) and (454.43,148.96) .. (449.77,148.96) .. controls (445.11,148.96) and (441.33,145.18) .. (441.33,140.52) -- cycle ;

\draw  [color={rgb, 255:red, 255; green, 255; blue, 255 }  ,draw opacity=1 ][fill={rgb, 255:red, 255; green, 255; blue, 255 }  ,fill opacity=1 ] (469.91,140.52) .. controls (469.91,135.86) and (473.69,132.08) .. (478.36,132.08) .. controls (483.02,132.08) and (486.8,135.86) .. (486.8,140.52) .. controls (486.8,145.18) and (483.02,148.96) .. (478.36,148.96) .. controls (473.69,148.96) and (469.91,145.18) .. (469.91,140.52) -- cycle ;

\draw   (530.27,35.49) -- (529.26,248.88) .. controls (529.19,264) and (493.56,276.25) .. (449.67,276.25) .. controls (405.78,276.25) and (370.27,264) .. (370.34,248.88) -- (371.34,35.49) .. controls (371.41,20.37) and (407.05,8.12) .. (450.94,8.12) .. controls (494.82,8.12) and (530.34,20.37) .. (530.27,35.49) .. controls (530.2,50.61) and (494.56,62.86) .. (450.68,62.86) .. controls (406.79,62.86) and (371.27,50.61) .. (371.34,35.49) ;
\draw   (399.06,28.76) .. controls (399.06,21.95) and (414.66,16.42) .. (433.9,16.42) .. controls (453.15,16.42) and (468.74,21.95) .. (468.74,28.76) .. controls (468.74,35.58) and (453.15,41.11) .. (433.9,41.11) .. controls (414.66,41.11) and (399.06,35.58) .. (399.06,28.76) -- cycle ;
\draw   (397.82,116.79) .. controls (397.82,109.97) and (413.42,104.45) .. (432.66,104.45) .. controls (451.9,104.45) and (467.5,109.97) .. (467.5,116.79) .. controls (467.5,123.61) and (451.9,129.14) .. (432.66,129.14) .. controls (413.42,129.14) and (397.82,123.61) .. (397.82,116.79) -- cycle ;
\draw  [color={rgb, 255:red, 255; green, 255; blue, 255 }  ,draw opacity=1 ][fill={rgb, 255:red, 255; green, 255; blue, 255 }  ,fill opacity=1 ] (410.33,105.52) .. controls (410.33,100.86) and (414.11,97.08) .. (418.77,97.08) .. controls (423.43,97.08) and (427.21,100.86) .. (427.21,105.52) .. controls (427.21,110.18) and (423.43,113.96) .. (418.77,113.96) .. controls (414.11,113.96) and (410.33,110.18) .. (410.33,105.52) -- cycle ;

\draw  [color={rgb, 255:red, 255; green, 255; blue, 255 }  ,draw opacity=1 ][fill={rgb, 255:red, 255; green, 255; blue, 255 }  ,fill opacity=1 ] (438.91,105.52) .. controls (438.91,100.86) and (442.69,97.08) .. (447.36,97.08) .. controls (452.02,97.08) and (455.8,100.86) .. (455.8,105.52) .. controls (455.8,110.18) and (452.02,113.96) .. (447.36,113.96) .. controls (442.69,113.96) and (438.91,110.18) .. (438.91,105.52) -- cycle ;

\draw   (431.38,80.97) .. controls (431.38,74.15) and (446.98,68.62) .. (466.22,68.62) .. controls (485.46,68.62) and (501.06,74.15) .. (501.06,80.97) .. controls (501.06,87.78) and (485.46,93.31) .. (466.22,93.31) .. controls (446.98,93.31) and (431.38,87.78) .. (431.38,80.97) -- cycle ;
\draw  [color={rgb, 255:red, 255; green, 255; blue, 255 }  ,draw opacity=1 ][fill={rgb, 255:red, 255; green, 255; blue, 255 }  ,fill opacity=1 ] (440.16,72.8) .. controls (440.16,68.14) and (443.94,64.36) .. (448.6,64.36) .. controls (453.26,64.36) and (457.04,68.14) .. (457.04,72.8) .. controls (457.04,77.46) and (453.26,81.24) .. (448.6,81.24) .. controls (443.94,81.24) and (440.16,77.46) .. (440.16,72.8) -- cycle ;
\draw  [color={rgb, 255:red, 255; green, 255; blue, 255 }  ,draw opacity=1 ][fill={rgb, 255:red, 255; green, 255; blue, 255 }  ,fill opacity=1 ] (469.99,70.32) .. controls (469.99,65.65) and (473.77,61.88) .. (478.43,61.88) .. controls (483.09,61.88) and (486.87,65.65) .. (486.87,70.32) .. controls (486.87,74.98) and (483.09,78.76) .. (478.43,78.76) .. controls (473.77,78.76) and (469.99,74.98) .. (469.99,70.32) -- cycle ;
\draw  [dash pattern={on 4.5pt off 4.5pt}] (428.54,28.19) -- (424.86,257.45) .. controls (424.85,258.44) and (422.14,259.21) .. (418.82,259.15) .. controls (415.5,259.1) and (412.82,258.25) .. (412.83,257.25) -- (416.51,28) .. controls (416.53,27) and (419.24,26.24) .. (422.56,26.29) .. controls (425.88,26.34) and (428.56,27.19) .. (428.54,28.19) .. controls (428.53,29.19) and (425.82,29.95) .. (422.5,29.9) .. controls (419.18,29.85) and (416.5,28.99) .. (416.51,28) ;
\draw  [dash pattern={on 4.5pt off 4.5pt}] (484.56,29.14) -- (480.86,259.71) .. controls (480.84,260.56) and (478.54,261.21) .. (475.71,261.17) .. controls (472.88,261.12) and (470.6,260.4) .. (470.61,259.55) -- (474.31,28.97) .. controls (474.33,28.12) and (476.63,27.47) .. (479.46,27.52) .. controls (482.29,27.56) and (484.57,28.29) .. (484.56,29.14) .. controls (484.55,29.99) and (482.24,30.64) .. (479.41,30.59) .. controls (476.58,30.55) and (474.3,29.82) .. (474.31,28.97) ;
\draw  [dash pattern={on 4.5pt off 4.5pt}] (455.58,28.04) -- (451.86,259.59) .. controls (451.85,260.5) and (449.35,261.21) .. (446.28,261.16) .. controls (443.22,261.11) and (440.75,260.33) .. (440.76,259.41) -- (444.48,27.86) .. controls (444.5,26.94) and (446.99,26.23) .. (450.06,26.28) .. controls (453.12,26.33) and (455.59,27.12) .. (455.58,28.04) .. controls (455.56,28.96) and (453.07,29.66) .. (450,29.61) .. controls (446.94,29.56) and (444.47,28.78) .. (444.48,27.86) ;

\draw (147.33,132.36) node   {$\alpha $};
\draw (204.96,138.38) node   {$\beta $};
\draw (370.85,32.71) node   {$\alpha \times \{0\}$};
\draw (371.19,112.48) node   {$\alpha \times \left\{\frac{2}{6}\right\}$};
\draw (369.87,183.08) node   {$\alpha \times \left\{\frac{4}{6}\right\}$};
\draw (530.66,77.95) node   {$\gamma \times \left\{\frac{1}{6}\right\}$};
\draw (532.66,148.95) node   {$\gamma \times \left\{\frac{3}{6}\right\}$};
\draw (530.66,217.95) node   {$\gamma \times \left\{\frac{5}{6}\right\}$};
\draw (368.06,242.59) node   {$\alpha \times \{1\}$};

\end{tikzpicture}

\caption{Cut and glue $A_k$ to $M_f\backslash ((\alpha \cup \beta) \times \{1/2\})$ when $k=3$}  
\end{figure}
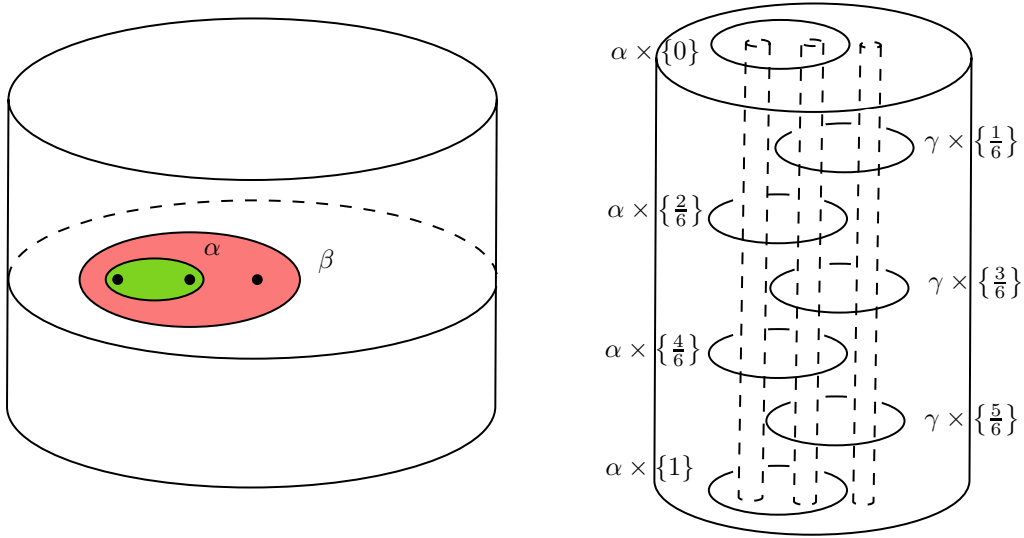

\end{proof}

\begin{prop}
Given $k$, there exists $B_k$, such that if $u_i,v_i>B_k$, then $h_kf$ is pseudo-Anosov and $\vol(M_{h_kf})\geq 3kV_8$.
\end{prop}
\begin{proof}
Let $M=M_f\backslash N(L_k)$. Let $\beta=\{\frac{1}{v_k}, \frac{1}{u_k},\dots, \frac{1}{v_1},\frac{1}{u_1} \}$, then by Theorem \ref{thurston}, $M_{h_kf}=M_\beta$, and when $u_i,v_i$ are big enough, the volume is approximatly equal to $\vol(M_f\backslash N(L_k))$. In particular, if $u_i,v_i$ are large enough, 
\[\vol(int(M_{h_kf}))\geq \vol(int(M_f\backslash N(L_k)))-kV_8 \geq 3kV_8\]
by Lemma 2.
\end{proof}

\begin{lem}
For $n,m>3$, $M_{h_kf^3_{n,m}}\cong M_{h_kf^3_{n+3,m}}\cong M_{h_kf^3_{n,m+3}}$.
\end{lem}
\begin{proof}
By Proposition 1, $\interior{M}=M_{h_kf}=M_{h_kf^3_{n,m}}$ is hyperbolic. Let $\Sigma_1$ be the subsurface in $S_{0,n+m+2}$ shown in Figure 3 containing 3 punctures, and let $\tau_1$ and $\tau_2$ denote the two components of $\partial\Sigma_1$, where $\tau_1$ and $\tau_2$ are two arcs connecting $x$ and $z$, with $\tau_2=f^3_{n,m}(\tau_1)$.

Construct a surface $\Sigma_2\subset M$ as follows. First, define a map 
\[\eta=(\eta_1,\eta_2):\Sigma_1\rightarrow S\times[0,1]\] 
so that $\eta(\Sigma_1)\cap S\times\{0\}=\tau_2\times\{0\}$, $\eta(\Sigma_1)\cap S\times\{1\}=\tau_1\times\{1\}$ and $\eta_1$ is the inclusion of $\Sigma_1$ into $S$. Since $f(\tau_1)=\tau_2$, if we project $p:S\times[0,1] \rightarrow M_f$, $\eta$ defines an embedding of $\Sigma_1/(\tau_1\isEquivTo{f} \tau_2)$, that is, $\Sigma_1$ with $\tau_1$ glued to $\tau_2$ by $f$. Since $\eta_1$ is the inclusion, $\Sigma_2=p\circ \eta(\Sigma_1/\tau_1\isEquivTo{f} \tau_2)$ is transverse to the suspension flow. By Theorem \ref{fried}, $[\Sigma_2] \in \overline{F\cdot \mathbb{R}^+}$.

We will define a surface $S\textprime$ such that $[S\textprime]=[S]+[\Sigma_2]$ in $H^1(M_f)$ as follows. Let $S_{\tau_2}$ denote the surface obtained by cutting $S$ along $\tau_2$. Then $S_{\tau_2}$ has two boundary components, denote $\tau^+_2, \tau^-_2$. Since $\tau_2=p\circ\eta(\Sigma_1)$, and $p\circ\eta(\tau_1)=p\circ\eta(\tau_2)=\tau_2 \subset S \subset M_f$, we can construct $S\textprime$ in $M_f$ by gluing $\tau^+_2$ to $\eta(\tau_2)$ and $\tau^-_2$ to $\eta(\tau_1)$, perturbed slightly to be embedded. Then $[S\textprime]=[S]+[\Sigma_2]$ and $S\textprime \pitchfork \Psi$. So $S\textprime$ is a fiber representing a class in $F\cdot \mathbb{R}^+ \subset H^1(M)$. By Theorem \ref{fried}, the first return map of $\psi$ is the monodromy $f\textprime:S\textprime \rightarrow S\textprime$. This is given by 
\[
f\textprime(x)=
 \begin{cases} 
   \eta(x) & \text{if } x \in \Sigma_1 \\
   f\circ\eta^{-1}(x)       & \text{if } x \in \eta(\Sigma_1)\\
   f(x) & \text{otherwise}
  \end{cases}
\]
See Figure 5. As indicated by Figure 6, $S\textprime \cong S_{0,n+m+5}$, and up to conjugation, $f\textprime=f^3_{n+3,m}$. Therefore, $M_{h_kf^3_{n,m}}\cong M_{h_kf^3_{n+3,m}}$. Similarly, if we pick another subsurface in $Y$ homeomorphic to $\Sigma_0$, one can show $M_{h_kf^3_{n,m}}\cong M_{h_kf^3_{n,m+3}}$.

\begin{figure}[H]
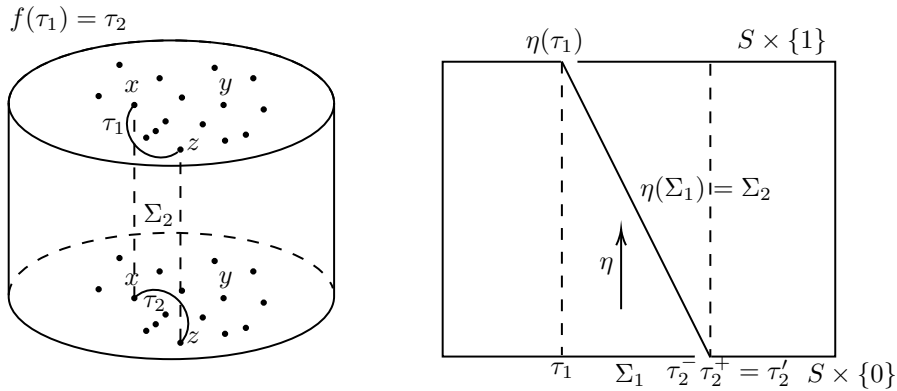
\centering

\tikzset{every picture/.style={line width=0.75pt}} 

\tikzset{every picture/.style={line width=0.75pt}} 



\caption{Obtain $\Sigma_2$ from $\eta:\Sigma_1\rightarrow S\times[0,1]$ and $S\textprime$ from $S$ and $\Sigma_2$ as shown.}  
\end{figure}

\begin{figure}[H]\centering

\tikzset{every picture/.style={line width=0.75pt}} 

%

\caption{Left: $S$. Right: $S\textprime$}  
\end{figure}

\end{proof}

\begin{lem}
For fixed $k$, and fixed $u_i,v_i\geq B_k$ (the constant from Proposition 3), there exists $R>0$ so that if $n=m\geq R$, then $h_kf^3_{n,n}: S_{0,2n+2}\to S_{0,2n+2}$ has $\log\lambda(h_kf^3_{n,n})\leq 54\frac{\log 2n+2}{2n+2}$.
\end{lem}
\begin{proof}
We can get the spine G as in Figure 7 on $S_{0,n+m+2}$. This is in fact a train track for $f_{n,m}$, as described in \cite{hironaka}, and hence also $f$. Then $f$ induces a map $g:G\to G$.

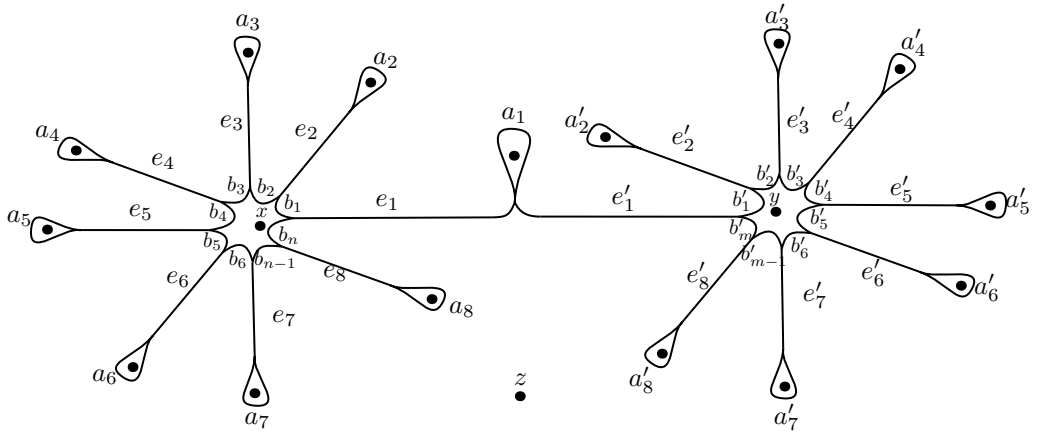
\begin{figure}[!ht]\centering

\tikzset{every picture/.style={line width=0.75pt}} 

\begin{tikzpicture}[x=0.75pt,y=0.75pt,yscale=-0.8,xscale=0.8]

\draw    (177.52,207.36) .. controls (171.45,216.18) and (175.89,220.15) .. (186.21,221.45) ;

\draw    (186.21,221.45) .. controls (168.28,223.86) and (164.37,231.88) .. (178.97,239.38) ;

\draw    (178.97,239.38) .. controls (166.6,238.11) and (163.04,238.1) .. (160.15,249.63) ;

\draw    (158.7,202.24) .. controls (161.6,213.76) and (167.39,215.05) .. (177.52,207.36) ;

\draw    (177.52,207.36) -- (220.95,154.85) ;

\draw    (186.21,221.45) -- (310.71,220.92) ;

\draw    (178.97,239.38) -- (247.01,263.72) ;

\draw    (160.15,249.63) -- (161.6,308.55) ;

\draw    (158.7,202.24) -- (157.25,138.2) ;

\draw    (157.25,138.2) .. controls (159.14,126.99) and (138.37,106.11) .. (156.31,106.95) .. controls (174.24,107.78) and (158.2,120.31) .. (157.25,138.2) -- cycle ;

\draw    (220.95,154.85) .. controls (229.41,146.34) and (224.77,118.83) .. (239.32,128.14) .. controls (253.87,137.45) and (232.71,140.27) .. (220.95,154.85) -- cycle ;

\draw    (247.01,263.72) .. controls (258.25,269.15) and (287.57,259.12) .. (280.49,273.72) .. controls (273.41,288.33) and (265.66,270.68) .. (247.01,263.72) -- cycle ;

\draw    (161.6,308.54) .. controls (159.85,319.77) and (180.87,340.44) .. (162.93,339.78) .. controls (144.98,339.12) and (160.87,326.44) .. (161.6,308.54) -- cycle ;

\draw    (310.71,220.92) .. controls (319.56,220.76) and (322.28,220.44) .. (325.55,211.31) ;

\draw    (474.96,234.62) .. controls (481.09,225.84) and (476.67,221.84) .. (466.36,220.49) ;

\draw    (466.36,220.49) .. controls (484.3,218.17) and (488.27,210.17) .. (473.71,202.6) ;

\draw    (473.71,202.6) .. controls (486.07,203.93) and (489.63,203.96) .. (492.6,192.44) ;

\draw    (493.75,239.84) .. controls (490.92,228.3) and (485.14,226.99) .. (474.96,234.62) ;

\draw    (474.96,234.62) -- (431.2,286.92) ;

\draw    (466.36,220.49) -- (341.85,220.41) ;

\draw    (473.71,202.6) -- (405.82,177.92) ;

\draw    (492.6,192.44) -- (491.52,133.52) ;

\draw    (493.75,239.84) -- (494.79,303.89) ;

\draw    (494.79,303.89) .. controls (492.83,315.08) and (513.47,336.06) .. (495.54,335.14) .. controls (477.61,334.22) and (493.73,321.77) .. (494.79,303.89) -- cycle ;

\draw    (431.2,286.92) .. controls (422.69,295.39) and (427.16,322.92) .. (412.66,313.54) .. controls (398.16,304.15) and (419.34,301.44) .. (431.2,286.92) -- cycle ;

\draw    (405.82,177.92) .. controls (394.62,172.43) and (365.24,182.32) .. (372.41,167.75) .. controls (379.58,153.19) and (387.22,170.87) .. (405.82,177.92) -- cycle ;

\draw    (491.52,133.52) .. controls (493.34,122.3) and (472.45,101.52) .. (490.39,102.28) .. controls (508.33,103.03) and (492.36,115.63) .. (491.52,133.52) -- cycle ;

\draw    (341.85,220.41) .. controls (333,220.52) and (328.45,217.71) .. (325.55,211.31) ;

\draw    (325.55,211.31) .. controls (327.98,194.74) and (301.23,163.86) .. (324.34,165.09) .. controls (347.45,166.33) and (326.77,184.85) .. (325.55,211.31) -- cycle ;

\draw  [color={rgb, 255:red, 0; green, 0; blue, 0 }  ][line width=3.75] [line join = round][line cap = round] (234.47,135.84) .. controls (234.47,135.84) and (234.47,135.84) .. (234.47,135.84) ;
\draw  [color={rgb, 255:red, 0; green, 0; blue, 0 }  ][line width=3.75] [line join = round][line cap = round] (324.95,182.02) .. controls (324.95,182.02) and (324.95,182.02) .. (324.95,182.02) ;
\draw  [color={rgb, 255:red, 0; green, 0; blue, 0 }  ][line width=3.75] [line join = round][line cap = round] (157.19,116.99) .. controls (157.19,116.99) and (157.19,116.99) .. (157.19,116.99) ;
\draw  [color={rgb, 255:red, 0; green, 0; blue, 0 }  ][line width=3.75] [line join = round][line cap = round] (273.11,272.5) .. controls (273.11,272.5) and (273.11,272.5) .. (273.11,272.5) ;
\draw  [color={rgb, 255:red, 0; green, 0; blue, 0 }  ][line width=3.75] [line join = round][line cap = round] (161.43,331.88) .. controls (161.43,331.88) and (161.43,331.88) .. (161.43,331.88) ;
\draw  [color={rgb, 255:red, 0; green, 0; blue, 0 }  ][line width=3.75] [line join = round][line cap = round] (164.73,226.32) .. controls (164.73,226.32) and (164.73,226.32) .. (164.73,226.32) ;
\draw  [color={rgb, 255:red, 0; green, 0; blue, 0 }  ][line width=3.75] [line join = round][line cap = round] (491.77,111.33) .. controls (491.77,111.33) and (491.77,111.33) .. (491.77,111.33) ;
\draw  [color={rgb, 255:red, 0; green, 0; blue, 0 }  ][line width=3.75] [line join = round][line cap = round] (382.44,170.24) .. controls (382.44,170.24) and (382.44,170.24) .. (382.44,170.24) ;
\draw  [color={rgb, 255:red, 0; green, 0; blue, 0 }  ][line width=3.75] [line join = round][line cap = round] (489.89,217.36) .. controls (489.89,217.36) and (489.89,217.36) .. (489.89,217.36) ;
\draw  [color={rgb, 255:red, 0; green, 0; blue, 0 }  ][line width=3.75] [line join = round][line cap = round] (418.26,306.9) .. controls (418.26,306.9) and (418.26,306.9) .. (418.26,306.9) ;
\draw  [color={rgb, 255:red, 0; green, 0; blue, 0 }  ][line width=3.75] [line join = round][line cap = round] (495.54,327.64) .. controls (495.54,327.64) and (495.54,327.64) .. (495.54,327.64) ;
\draw  [color={rgb, 255:red, 0; green, 0; blue, 0 }  ][line width=3.75] [line join = round][line cap = round] (328.72,333.76) .. controls (328.72,333.76) and (328.72,333.76) .. (328.72,333.76) ;
\draw    (511.17,198.88) .. controls (505.09,207.7) and (509.53,211.67) .. (519.85,212.97) ;

\draw    (519.85,212.97) .. controls (501.93,215.38) and (498.01,223.4) .. (512.62,230.9) ;

\draw    (512.62,230.9) .. controls (500.24,229.62) and (496.69,229.62) .. (493.79,241.14) ;

\draw    (492.35,193.75) .. controls (495.24,205.28) and (501.03,206.56) .. (511.17,198.88) ;

\draw    (511.17,198.88) -- (554.6,146.36) ;

\draw    (519.85,212.97) -- (609.67,213.36) ;

\draw    (512.62,230.9) -- (580.66,255.23) ;

\draw    (554.6,146.36) .. controls (563.06,137.86) and (558.41,110.35) .. (572.97,119.66) .. controls (587.52,128.97) and (566.36,131.79) .. (554.6,146.36) -- cycle ;

\draw    (580.66,255.23) .. controls (591.89,260.67) and (621.21,250.64) .. (614.13,265.24) .. controls (607.06,279.84) and (599.31,262.19) .. (580.66,255.23) -- cycle ;

\draw  [color={rgb, 255:red, 0; green, 0; blue, 0 }  ][line width=3.75] [line join = round][line cap = round] (568.12,127.36) .. controls (568.12,127.36) and (568.12,127.36) .. (568.12,127.36) ;
\draw  [color={rgb, 255:red, 0; green, 0; blue, 0 }  ][line width=3.75] [line join = round][line cap = round] (606.76,264.02) .. controls (606.76,264.02) and (606.76,264.02) .. (606.76,264.02) ;
\draw    (603.99,213.25) .. controls (615.17,215.27) and (636.3,194.75) .. (635.25,212.68) .. controls (634.19,230.6) and (621.86,214.41) .. (603.99,213.25) -- cycle ;

\draw  [color={rgb, 255:red, 0; green, 0; blue, 0 }  ][line width=3.75] [line join = round][line cap = round] (625.19,213.43) .. controls (625.19,213.43) and (625.19,213.43) .. (625.19,213.43) ;
\draw    (141.31,243.11) .. controls (147.44,234.32) and (143.03,230.33) .. (132.71,228.97) ;

\draw    (132.71,228.97) .. controls (150.65,226.65) and (154.62,218.65) .. (140.07,211.08) ;

\draw    (140.07,211.08) .. controls (152.43,212.41) and (155.98,212.44) .. (158.95,200.93) ;

\draw    (160.1,248.32) .. controls (157.28,236.78) and (151.49,235.47) .. (141.31,243.11) ;

\draw    (141.31,243.11) -- (97.55,295.4) ;

\draw    (132.71,228.97) -- (46.99,228.91) ;

\draw    (140.07,211.08) -- (72.18,186.41) ;

\draw    (97.55,295.4) .. controls (89.04,303.87) and (93.51,331.4) .. (79.01,322.02) .. controls (64.52,312.64) and (85.7,309.92) .. (97.55,295.4) -- cycle ;

\draw    (72.18,186.41) .. controls (60.98,180.92) and (31.6,190.8) .. (38.77,176.24) .. controls (45.94,161.67) and (53.57,179.36) .. (72.18,186.41) -- cycle ;

\draw  [color={rgb, 255:red, 0; green, 0; blue, 0 }  ][line width=3.75] [line join = round][line cap = round] (48.8,178.72) .. controls (48.8,178.72) and (48.8,178.72) .. (48.8,178.72) ;
\draw  [color={rgb, 255:red, 0; green, 0; blue, 0 }  ][line width=3.75] [line join = round][line cap = round] (84.61,315.38) .. controls (84.61,315.38) and (84.61,315.38) .. (84.61,315.38) ;
\draw    (51.86,229) .. controls (40.7,226.87) and (19.38,247.2) .. (20.6,229.28) .. controls (21.81,211.37) and (34,227.68) .. (51.86,229) -- cycle ;

\draw  [color={rgb, 255:red, 0; green, 0; blue, 0 }  ][line width=3.75] [line join = round][line cap = round] (30.65,228.62) .. controls (30.65,228.62) and (30.65,228.62) .. (30.65,228.62) ;

\draw (245.31,211.24) node   {$e_{1}$};
\draw (194.41,166.94) node   {$e_{2}$};
\draw (147.29,161.29) node   {$e_{3}$};
\draw (212.32,256.48) node   {$e_{8}$};
\draw (180.28,284.75) node   {$e_{7}$};
\draw (394.23,208.41) node   {$e'_{1}$};
\draw (431.93,170.71) node   {$e'_{2}$};
\draw (504.5,159.4) node   {$e'_{3}$};
\draw (514.87,269.67) node   {$e'_{7}$};
\draw (441.35,257.42) node   {$e'_{8}$};
\draw (325.42,156.57) node   {$a_{1}$};
\draw (244.37,121.7) node   {$a_{2}$};
\draw (157.66,98.14) node   {$a_{3}$};
\draw (163.31,348.84) node   {$a_{7}$};
\draw (292.43,279.1) node   {$a_{8}$};
\draw (185.93,212.18) node [scale=0.8]  {$b_{1}$};
\draw (168.97,201.81) node [scale=0.8]  {$b_{2}$};
\draw (183.1,231.97) node [scale=0.8]  {$b_{n}$};
\draw (174.62,248) node [scale=0.8]  {$b_{n-1}$};
\draw (365.01,159.4) node   {$a'_{2}$};
\draw (491.3,94.37) node   {$a'_{3}$};
\draw (405.54,323.4) node   {$a'_{8}$};
\draw (496.96,346.96) node   {$a'_{7}$};
\draw (166.14,217.84) node [scale=0.8]  {$x$};
\draw (489.42,209.35) node [scale=0.8]  {$y$};
\draw (468.68,210.3) node [scale=0.8]  {$b'_{1}$};
\draw (483.76,194.27) node [scale=0.8]  {$b'_{2}$};
\draw (468.68,228.2) node [scale=0.8]  {$b'_{m}$};
\draw (482.82,245.17) node [scale=0.8]  {$b'_{m-1}$};
\draw (329.19,323.4) node   {$z$};
\draw (532.77,156.57) node   {$e'_{4}$};
\draw (520.52,203.7) node [scale=0.8]  {$b'_{4}$};
\draw (502.61,194.27) node [scale=0.8]  {$b'_{3}$};
\draw (577.07,110.39) node   {$a'_{4}$};
\draw (643.05,210.3) node   {$a'_{5}$};
\draw (623.25,264.96) node   {$a'_{6}$};
\draw (517.69,221.61) node [scale=0.8]  {$b'_{5}$};
\draw (505.44,239.51) node [scale=0.8]  {$b'_{6}$};
\draw (569.53,202.76) node   {$e'_{5}$};
\draw (551.62,255.54) node   {$e'_{6}$};
\draw (67.18,321.51) node   {$a_{6}$};
\draw (13.45,221.61) node   {$a_{5}$};
\draw (113.36,260.25) node   {$e_{6}$};
\draw (89.8,217.84) node   {$e_{5}$};
\draw (151.06,247.05) node [scale=0.8]  {$b_{6}$};
\draw (135.04,235.74) node [scale=0.8]  {$b_{5}$};
\draw (138.81,217.84) node [scale=0.8]  {$b_{4}$};
\draw (150.12,200.87) node [scale=0.8]  {$b_{3}$};
\draw (103.93,185.79) node   {$e_{4}$};
\draw (31.36,167.88) node   {$a_{4}$};

\end{tikzpicture}

\caption{Spine of $S_{0,n+m+2}$ when $n=m=8$}  
\end{figure}

The graph $G$ contains the loop edges $a_1, a_2, \dots, a_n$, and $a\textprime_2, a\textprime_3, \dots, a\textprime_m$, which $g$ acts on as a permutation, and ``peripheral" edges $b_1, b_2, \dots, b_n$, and $b\textprime_1, b\textprime_2, \dots, b\textprime_m$, which $g$ also acts on them as a permutation. The transition matrix has the following form: 
\[
T=
\left[
\begin{array}{c|c}
A & *\\
\hline
0 & P\\
\end{array}
\right]
\]
where $P$ corresponds to $e_1, e_2, \dots, e_n$, $e\textprime_1, e\textprime_2, \dots, e\textprime_m$. The matrix $A$ is a permutation matrix corresponds to $a_1, a_2, \dots, a_n$,  $a\textprime_1, a\textprime_2, \dots, a\textprime_m$, $b_1, b_2, \dots, b_n$, $b\textprime_1, b\textprime_2, \dots, b\textprime_m$. So the largest eigenvalue of $T$ (in absolute value) will be the largest eigenvalue of $P$.
If we remove all the non-contributing edges, we have 
\[
\begin{array}{rcl} 
e_i & \to & e_{i+3} \mbox{ }  \mbox{ for  } 1\leq i\leq n-3 \\ 
e\textprime_i & \to & e\textprime_{i+3} \mbox{ } \mbox{ for  } 1< i\leq m-2  \\
e\textprime_1 & \to & e\textprime_4e\textprime_4e\textprime_3e\textprime_3e\textprime_2e\textprime_2e\textprime_1e_1e_2e_2e_3e_3e_4  \\
e_n & \to & e_3e_3e_2e_2e_1e\textprime_1e\textprime_2e\textprime_2e\textprime_3e\textprime_3e\textprime_4   \\
e\textprime_m & \to & e\textprime_3e\textprime_3e\textprime_2e\textprime_2e\textprime_1e_1e_2e_2e_3  \\
e_{n-1} & \to & e_2e_2e_1e\textprime_1e\textprime_2e\textprime_2e\textprime_3  \\
e\textprime_{m-1} & \to & e\textprime_2e\textprime_2e\textprime_1e_1e_2 \\
e_{n-2} & \to & e_1e\textprime_1e\textprime_2 
\end{array}
\]

\begin{figure}[!ht]\centering

\tikzset{every picture/.style={line width=0.75pt}} 

\begin{tikzpicture}[x=0.75pt,y=0.75pt,yscale=-1,xscale=1]

\draw   (7,150.45) .. controls (7,81.01) and (105.56,24.73) .. (227.13,24.73) .. controls (348.71,24.73) and (447.27,81.01) .. (447.27,150.45) .. controls (447.27,219.88) and (348.71,276.17) .. (227.13,276.17) .. controls (105.56,276.17) and (7,219.88) .. (7,150.45) -- cycle ;
\draw   (31.29,150.45) .. controls (31.29,88.68) and (118.97,38.6) .. (227.13,38.6) .. controls (335.3,38.6) and (422.98,88.68) .. (422.98,150.45) .. controls (422.98,212.22) and (335.3,262.3) .. (227.13,262.3) .. controls (118.97,262.3) and (31.29,212.22) .. (31.29,150.45) -- cycle ;
\draw   (54.74,150.45) .. controls (54.74,96.07) and (131.92,51.99) .. (227.13,51.99) .. controls (322.34,51.99) and (399.53,96.07) .. (399.53,150.45) .. controls (399.53,204.82) and (322.34,248.9) .. (227.13,248.9) .. controls (131.92,248.9) and (54.74,204.82) .. (54.74,150.45) -- cycle ;
\draw   (80.28,150.45) .. controls (80.28,104.13) and (146.03,66.58) .. (227.13,66.58) .. controls (308.23,66.58) and (373.98,104.13) .. (373.98,150.45) .. controls (373.98,196.77) and (308.23,234.31) .. (227.13,234.31) .. controls (146.03,234.31) and (80.28,196.77) .. (80.28,150.45) -- cycle ;
\draw   (106.42,150.45) .. controls (106.42,112.37) and (160.47,81.51) .. (227.13,81.51) .. controls (293.8,81.51) and (347.85,112.37) .. (347.85,150.45) .. controls (347.85,188.52) and (293.8,219.39) .. (227.13,219.39) .. controls (160.47,219.39) and (106.42,188.52) .. (106.42,150.45) -- cycle ;
\draw   (133.69,150.45) .. controls (133.69,120.98) and (175.53,97.08) .. (227.13,97.08) .. controls (278.74,97.08) and (320.57,120.98) .. (320.57,150.45) .. controls (320.57,179.92) and (278.74,203.81) .. (227.13,203.81) .. controls (175.53,203.81) and (133.69,179.92) .. (133.69,150.45) -- cycle ;
\draw  [fill={rgb, 255:red, 255; green, 255; blue, 255 }  ,fill opacity=1 ] (194.83,7.83) -- (255.6,7.83) -- (255.6,114.17) -- (194.83,114.17) -- cycle ;
\draw  [color={rgb, 255:red, 0; green, 0; blue, 0 }  ][line width=2.25] [line join = round][line cap = round] (128.39,37.65) .. controls (128.39,37.65) and (128.39,37.65) .. (128.39,37.65) ;
\draw  [color={rgb, 255:red, 0; green, 0; blue, 0 }  ][line width=2.25] [line join = round][line cap = round] (161.61,29.98) .. controls (161.61,29.98) and (161.61,29.98) .. (161.61,29.98) ;
\draw  [color={rgb, 255:red, 0; green, 0; blue, 0 }  ][line width=2.25] [line join = round][line cap = round] (163.31,44.46) .. controls (163.31,44.46) and (163.31,44.46) .. (163.31,44.46) ;
\draw  [color={rgb, 255:red, 0; green, 0; blue, 0 }  ][line width=2.25] [line join = round][line cap = round] (130.09,52.98) .. controls (130.09,52.98) and (130.09,52.98) .. (130.09,52.98) ;
\draw  [color={rgb, 255:red, 0; green, 0; blue, 0 }  ][line width=2.25] [line join = round][line cap = round] (165.02,58.09) .. controls (165.02,58.09) and (165.02,58.09) .. (165.02,58.09) ;
\draw  [color={rgb, 255:red, 0; green, 0; blue, 0 }  ][line width=2.25] [line join = round][line cap = round] (133.5,67.46) .. controls (133.5,67.46) and (133.5,67.46) .. (133.5,67.46) ;
\draw  [color={rgb, 255:red, 0; green, 0; blue, 0 }  ][line width=2.25] [line join = round][line cap = round] (165.87,74.28) .. controls (165.87,74.28) and (165.87,74.28) .. (165.87,74.28) ;
\draw  [color={rgb, 255:red, 0; green, 0; blue, 0 }  ][line width=2.25] [line join = round][line cap = round] (136.91,84.5) .. controls (136.91,84.5) and (136.91,84.5) .. (136.91,84.5) ;
\draw  [color={rgb, 255:red, 0; green, 0; blue, 0 }  ][line width=2.25] [line join = round][line cap = round] (166.72,90.46) .. controls (166.72,90.46) and (166.72,90.46) .. (166.72,90.46) ;
\draw  [color={rgb, 255:red, 0; green, 0; blue, 0 }  ][line width=2.25] [line join = round][line cap = round] (142.02,101.54) .. controls (142.02,101.54) and (142.02,101.54) .. (142.02,101.54) ;
\draw  [color={rgb, 255:red, 0; green, 0; blue, 0 }  ][line width=2.25] [line join = round][line cap = round] (169.28,108.35) .. controls (169.28,108.35) and (169.28,108.35) .. (169.28,108.35) ;
\draw  [color={rgb, 255:red, 0; green, 0; blue, 0 }  ][line width=2.25] [line join = round][line cap = round] (145.43,124.54) .. controls (145.43,124.54) and (145.43,124.54) .. (145.43,124.54) ;
\draw  [color={rgb, 255:red, 0; green, 0; blue, 0 }  ][line width=2.25] [line join = round][line cap = round] (325.02,37.65) .. controls (325.02,37.65) and (325.02,37.65) .. (325.02,37.65) ;
\draw  [color={rgb, 255:red, 0; green, 0; blue, 0 }  ][line width=2.25] [line join = round][line cap = round] (291.8,29.98) .. controls (291.8,29.98) and (291.8,29.98) .. (291.8,29.98) ;
\draw  [color={rgb, 255:red, 0; green, 0; blue, 0 }  ][line width=2.25] [line join = round][line cap = round] (290.1,44.46) .. controls (290.1,44.46) and (290.1,44.46) .. (290.1,44.46) ;
\draw  [color={rgb, 255:red, 0; green, 0; blue, 0 }  ][line width=2.25] [line join = round][line cap = round] (323.32,52.98) .. controls (323.32,52.98) and (323.32,52.98) .. (323.32,52.98) ;
\draw  [color={rgb, 255:red, 0; green, 0; blue, 0 }  ][line width=2.25] [line join = round][line cap = round] (288.4,58.09) .. controls (288.4,58.09) and (288.4,58.09) .. (288.4,58.09) ;
\draw  [color={rgb, 255:red, 0; green, 0; blue, 0 }  ][line width=2.25] [line join = round][line cap = round] (319.91,67.46) .. controls (319.91,67.46) and (319.91,67.46) .. (319.91,67.46) ;
\draw  [color={rgb, 255:red, 0; green, 0; blue, 0 }  ][line width=2.25] [line join = round][line cap = round] (287.54,74.28) .. controls (287.54,74.28) and (287.54,74.28) .. (287.54,74.28) ;
\draw  [color={rgb, 255:red, 0; green, 0; blue, 0 }  ][line width=2.25] [line join = round][line cap = round] (316.51,83.65) .. controls (316.51,83.65) and (316.51,83.65) .. (316.51,83.65) ;
\draw  [color={rgb, 255:red, 0; green, 0; blue, 0 }  ][line width=2.25] [line join = round][line cap = round] (286.69,90.46) .. controls (286.69,90.46) and (286.69,90.46) .. (286.69,90.46) ;
\draw  [color={rgb, 255:red, 0; green, 0; blue, 0 }  ][line width=2.25] [line join = round][line cap = round] (311.4,101.54) .. controls (311.4,101.54) and (311.4,101.54) .. (311.4,101.54) ;
\draw  [color={rgb, 255:red, 0; green, 0; blue, 0 }  ][line width=2.25] [line join = round][line cap = round] (284.14,108.35) .. controls (284.14,108.35) and (284.14,108.35) .. (284.14,108.35) ;
\draw  [color={rgb, 255:red, 0; green, 0; blue, 0 }  ][line width=2.25] [line join = round][line cap = round] (308.84,124.54) .. controls (308.84,124.54) and (308.84,124.54) .. (308.84,124.54) ;
\draw   (174.54,25.04) -- (182.39,26.99) -- (176.12,32.11) ;
\draw   (139.61,31) -- (147.47,32.95) -- (141.19,38.07) ;
\draw   (174.54,39.52) -- (182.39,41.47) -- (176.12,46.59) ;
\draw   (139.61,46.34) -- (147.47,48.28) -- (141.19,53.4) ;
\draw   (174.54,53.15) -- (182.39,55.1) -- (176.12,60.22) ;
\draw   (143.02,59.97) -- (150.88,61.91) -- (144.6,67.03) ;
\draw   (174.54,68.48) -- (182.39,70.43) -- (176.12,75.55) ;
\draw   (145.15,76.99) -- (153.23,77.47) -- (147.99,83.64) ;
\draw   (175.39,84.67) -- (183.25,86.62) -- (176.97,91.74) ;
\draw   (175.13,102.16) -- (183.15,103.3) -- (177.43,109.03) ;
\draw   (147.65,94.15) -- (155.74,94.42) -- (150.67,100.73) ;
\draw   (151.68,114.6) -- (159.69,113.47) -- (155.79,120.56) ;
\draw   (268.07,24.18) -- (275.33,27.76) -- (268.12,31.42) ;
\draw   (305.31,29.62) -- (311.92,34.29) -- (304.22,36.78) ;
\draw   (269.45,37.35) -- (276.15,41.89) -- (268.5,44.52) ;
\draw   (267.75,50.98) -- (274.45,55.52) -- (266.8,58.15) ;
\draw   (266.9,66.31) -- (273.6,70.85) -- (265.94,73.49) ;
\draw   (269.45,82.49) -- (276.15,87.03) -- (268.5,89.67) ;
\draw   (268.29,99.17) -- (274.36,104.52) -- (266.43,106.17) ;
\draw   (304.9,43.81) -- (311,49.13) -- (303.08,50.82) ;
\draw   (301.52,57.42) -- (307.58,62.79) -- (299.65,64.42) ;
\draw   (299.1,73.53) -- (305,79.08) -- (297.02,80.47) ;
\draw   (297.86,90.35) -- (303.17,96.46) -- (295.1,97.04) ;
\draw   (294.79,109.82) -- (299.66,116.28) -- (291.56,116.3) ;

\draw (225.93,56.74) node   {$D$};
\draw (162.04,22.31) node [scale=0.7]  {$e'_{m-2}$};
\draw (124.56,29.98) node [scale=0.7]  {$e'_{m-5}$};
\draw (162.04,37.65) node [scale=0.7]  {$e'_{m-3}$};
\draw (126.26,46.17) node [scale=0.7]  {$e'_{m-6}$};
\draw (129.67,60.65) node [scale=0.7]  {$e'_{m-7}$};
\draw (161.19,52.13) node [scale=0.7]  {$e'_{m-4}$};
\draw (163.74,67.46) node [scale=0.7]  {$e_{n-3}$};
\draw (133.93,76.83) node [scale=0.7]  {$e_{n-6}$};
\draw (138.19,93.02) node [scale=0.7]  {$e_{n-7}$};
\draw (164.59,83.65) node [scale=0.7]  {$e_{n-4}$};
\draw (168,99.83) node [scale=0.7]  {$e_{n-5}$};
\draw (139.89,115.17) node [scale=0.7]  {$e_{n-8}$};
\draw (93.89,35.23) node [scale=0.9,rotate=-335.11]  {$\dotsc $};
\draw (99.85,52.27) node [scale=0.9,rotate=-335.11]  {$\dotsc $};
\draw (104.96,66.75) node [scale=0.9,rotate=-335.11]  {$\dotsc $};
\draw (110.07,82.09) node [scale=0.9,rotate=-335.11]  {$\dotsc $};
\draw (116.04,101.68) node [scale=0.9,rotate=-335.11]  {$\dotsc $};
\draw (125.41,124.68) node [scale=0.9,rotate=-335.11]  {$\dotsc $};
\draw (291.52,22.31) node [scale=0.7]  {$e'_{7}$};
\draw (325.59,29.98) node [scale=0.7]  {$e'_{10}$};
\draw (291.52,37.65) node [scale=0.7]  {$e'_{6}$};
\draw (324.74,47.02) node [scale=0.7]  {$e'_{9}$};
\draw (288.96,51.28) node [scale=0.7]  {$e'_{5}$};
\draw (322.19,60.65) node [scale=0.7]  {$e'_{8}$};
\draw (288.11,67.46) node [scale=0.7]  {$e_{6}$};
\draw (317.93,77.69) node [scale=0.7]  {$e_{9}$};
\draw (287.26,83.65) node [scale=0.7]  {$e_{5}$};
\draw (312.81,94.72) node [scale=0.7]  {$e_{8}$};
\draw (284.7,100.69) node [scale=0.7]  {$e_{7}$};
\draw (311.96,116.87) node [scale=0.7]  {$e_{10}$};
\draw (346.89,33.53) node [scale=0.9,rotate=-17.84]  {$\dotsc $};
\draw (343.48,50.57) node [scale=0.9,rotate=-17.84]  {$\dotsc $};
\draw (340.93,65.05) node [scale=0.9,rotate=-17.84]  {$\dotsc $};
\draw (334.96,82.94) node [scale=0.9,rotate=-17.84]  {$\dotsc $};
\draw (331.56,99.12) node [scale=0.9,rotate=-17.84]  {$\dotsc $};
\draw (329,122.12) node [scale=0.9,rotate=-17.84]  {$\dotsc $};

\end{tikzpicture}

\caption{The directed graph $\Gamma$ associated to $f$.}  
\end{figure}

\begin{figure}[!ht]\centering

\tikzset{every picture/.style={line width=0.75pt}} 

\begin{tikzpicture}[x=0.75pt,y=0.75pt,yscale=-1,xscale=1]

\draw    (87,329.55) -- (551.93,329.55) ;

\draw  [fill={rgb, 255:red, 0; green, 0; blue, 0 }  ,fill opacity=1 ] (305.67,86.22) .. controls (305.67,84.49) and (307.07,83.08) .. (308.8,83.08) .. controls (310.53,83.08) and (311.93,84.49) .. (311.93,86.22) .. controls (311.93,87.95) and (310.53,89.35) .. (308.8,89.35) .. controls (307.07,89.35) and (305.67,87.95) .. (305.67,86.22) -- cycle ;
\draw  [fill={rgb, 255:red, 0; green, 0; blue, 0 }  ,fill opacity=1 ] (451.67,86.22) .. controls (451.67,84.49) and (453.07,83.08) .. (454.8,83.08) .. controls (456.53,83.08) and (457.93,84.49) .. (457.93,86.22) .. controls (457.93,87.95) and (456.53,89.35) .. (454.8,89.35) .. controls (453.07,89.35) and (451.67,87.95) .. (451.67,86.22) -- cycle ;
\draw  [fill={rgb, 255:red, 0; green, 0; blue, 0 }  ,fill opacity=1 ] (183.67,329.22) .. controls (183.67,327.49) and (185.07,326.08) .. (186.8,326.08) .. controls (188.53,326.08) and (189.93,327.49) .. (189.93,329.22) .. controls (189.93,330.95) and (188.53,332.35) .. (186.8,332.35) .. controls (185.07,332.35) and (183.67,330.95) .. (183.67,329.22) -- cycle ;
\draw  [fill={rgb, 255:red, 0; green, 0; blue, 0 }  ,fill opacity=1 ] (306.67,329.22) .. controls (306.67,327.49) and (308.07,326.08) .. (309.8,326.08) .. controls (311.53,326.08) and (312.93,327.49) .. (312.93,329.22) .. controls (312.93,330.95) and (311.53,332.35) .. (309.8,332.35) .. controls (308.07,332.35) and (306.67,330.95) .. (306.67,329.22) -- cycle ;
\draw  [fill={rgb, 255:red, 0; green, 0; blue, 0 }  ,fill opacity=1 ] (431.67,329.22) .. controls (431.67,327.49) and (433.07,326.08) .. (434.8,326.08) .. controls (436.53,326.08) and (437.93,327.49) .. (437.93,329.22) .. controls (437.93,330.95) and (436.53,332.35) .. (434.8,332.35) .. controls (433.07,332.35) and (431.67,330.95) .. (431.67,329.22) -- cycle ;
\draw  [fill={rgb, 255:red, 0; green, 0; blue, 0 }  ,fill opacity=1 ] (185.67,131.22) .. controls (185.67,129.49) and (187.07,128.08) .. (188.8,128.08) .. controls (190.53,128.08) and (191.93,129.49) .. (191.93,131.22) .. controls (191.93,132.95) and (190.53,134.35) .. (188.8,134.35) .. controls (187.07,134.35) and (185.67,132.95) .. (185.67,131.22) -- cycle ;
\draw  [fill={rgb, 255:red, 0; green, 0; blue, 0 }  ,fill opacity=1 ] (186.67,177.22) .. controls (186.67,175.49) and (188.07,174.08) .. (189.8,174.08) .. controls (191.53,174.08) and (192.93,175.49) .. (192.93,177.22) .. controls (192.93,178.95) and (191.53,180.35) .. (189.8,180.35) .. controls (188.07,180.35) and (186.67,178.95) .. (186.67,177.22) -- cycle ;
\draw  [fill={rgb, 255:red, 0; green, 0; blue, 0 }  ,fill opacity=1 ] (186.67,225.22) .. controls (186.67,223.49) and (188.07,222.08) .. (189.8,222.08) .. controls (191.53,222.08) and (192.93,223.49) .. (192.93,225.22) .. controls (192.93,226.95) and (191.53,228.35) .. (189.8,228.35) .. controls (188.07,228.35) and (186.67,226.95) .. (186.67,225.22) -- cycle ;
\draw  [fill={rgb, 255:red, 0; green, 0; blue, 0 }  ,fill opacity=1 ] (186.67,277.22) .. controls (186.67,275.49) and (188.07,274.08) .. (189.8,274.08) .. controls (191.53,274.08) and (192.93,275.49) .. (192.93,277.22) .. controls (192.93,278.95) and (191.53,280.35) .. (189.8,280.35) .. controls (188.07,280.35) and (186.67,278.95) .. (186.67,277.22) -- cycle ;
\draw  [fill={rgb, 255:red, 0; green, 0; blue, 0 }  ,fill opacity=1 ] (418.67,132.22) .. controls (418.67,130.49) and (420.07,129.08) .. (421.8,129.08) .. controls (423.53,129.08) and (424.93,130.49) .. (424.93,132.22) .. controls (424.93,133.95) and (423.53,135.35) .. (421.8,135.35) .. controls (420.07,135.35) and (418.67,133.95) .. (418.67,132.22) -- cycle ;
\draw  [fill={rgb, 255:red, 0; green, 0; blue, 0 }  ,fill opacity=1 ] (421.67,176.08) .. controls (421.67,174.35) and (423.07,172.95) .. (424.8,172.95) .. controls (426.53,172.95) and (427.93,174.35) .. (427.93,176.08) .. controls (427.93,177.81) and (426.53,179.22) .. (424.8,179.22) .. controls (423.07,179.22) and (421.67,177.81) .. (421.67,176.08) -- cycle ;
\draw    (189.8,225.22) -- (454.8,86.22) ;

\draw  [fill={rgb, 255:red, 0; green, 0; blue, 0 }  ,fill opacity=1 ] (422.67,225.22) .. controls (422.67,223.49) and (424.07,222.08) .. (425.8,222.08) .. controls (427.53,222.08) and (428.93,223.49) .. (428.93,225.22) .. controls (428.93,226.95) and (427.53,228.35) .. (425.8,228.35) .. controls (424.07,228.35) and (422.67,226.95) .. (422.67,225.22) -- cycle ;
\draw  [fill={rgb, 255:red, 0; green, 0; blue, 0 }  ,fill opacity=1 ] (423.67,276.22) .. controls (423.67,274.49) and (425.07,273.08) .. (426.8,273.08) .. controls (428.53,273.08) and (429.93,274.49) .. (429.93,276.22) .. controls (429.93,277.95) and (428.53,279.35) .. (426.8,279.35) .. controls (425.07,279.35) and (423.67,277.95) .. (423.67,276.22) -- cycle ;
\draw    (308.8,86.22) -- (434.8,329.22) ;

\draw   (268.57,45.98) .. controls (268.57,23.76) and (286.58,5.75) .. (308.8,5.75) .. controls (331.02,5.75) and (349.03,23.76) .. (349.03,45.98) .. controls (349.03,68.2) and (331.02,86.22) .. (308.8,86.22) .. controls (286.58,86.22) and (268.57,68.2) .. (268.57,45.98) -- cycle ;
\draw    (188.8,131.22) -- (308.8,86.22) ;

\draw    (308.8,86.22) -- (189.8,177.22) ;

\draw    (189.8,225.22) -- (308.8,86.22) ;

\draw    (308.8,86.22) -- (189.8,277.22) ;

\draw    (186.8,329.22) -- (308.8,86.22) ;

\draw [color={rgb, 255:red, 255; green, 0; blue, 0 }  ,draw opacity=1 ][line width=2.25]    (308.8,86.22) -- (421.8,132.22) ;

\draw [color={rgb, 255:red, 255; green, 0; blue, 0 }  ,draw opacity=1 ][line width=2.25]    (308.8,86.22) -- (424.8,176.08) ;

\draw [color={rgb, 255:red, 255; green, 0; blue, 0 }  ,draw opacity=1 ][line width=2.25]    (308.8,86.22) -- (425.8,225.22) ;

\draw [color={rgb, 255:red, 255; green, 0; blue, 0 }  ,draw opacity=1 ][line width=1.5]    (426.8,276.22) -- (308.8,86.22) ;

\draw    (188.8,131.22) -- (309.8,329.22) ;

\draw    (189.8,177.22) -- (309.8,329.22) ;

\draw    (189.8,225.22) -- (309.8,329.22) ;

\draw    (189.8,277.22) -- (309.8,329.22) ;

\draw [color={rgb, 255:red, 255; green, 0; blue, 0 }  ,draw opacity=1 ][line width=2.25]    (188.8,131.22) -- (424.8,176.08) ;

\draw    (188.8,131.22) -- (425.8,225.22) ;

\draw [color={rgb, 255:red, 255; green, 0; blue, 0 }  ,draw opacity=1 ][line width=2.25]    (188.8,131.22) -- (426.8,276.22) ;

\draw [color={rgb, 255:red, 255; green, 0; blue, 0 }  ,draw opacity=1 ][line width=2.25]    (189.8,225.22) -- (421.8,132.22) ;

\draw [color={rgb, 255:red, 255; green, 0; blue, 0 }  ,draw opacity=1 ][line width=2.25]    (189.8,225.22) -- (424.8,176.08) ;

\draw [color={rgb, 255:red, 255; green, 0; blue, 0 }  ,draw opacity=1 ][line width=2.25]    (189.8,225.22) -- (308.46,250.75) -- (426.8,276.22) ;

\draw    (189.8,277.22) -- (421.8,132.22) ;

\draw [color={rgb, 255:red, 255; green, 0; blue, 0 }  ,draw opacity=1 ][line width=2.25]    (189.8,277.22) -- (424.8,176.08) ;

\draw    (186.8,329.22) -- (426.8,276.22) ;

\draw   (260.59,56.87) -- (268.59,40.33) -- (277.02,56.66) ;
\draw   (141.33,79.75) -- (156,87.08) -- (141.33,94.42) ;
\draw   (123.33,123.75) -- (138,131.08) -- (123.33,138.42) ;
\draw   (121.33,168.75) -- (136,176.08) -- (121.33,183.42) ;
\draw   (123.33,217.75) -- (138,225.08) -- (123.33,232.42) ;
\draw   (120.33,269.75) -- (135,277.08) -- (120.33,284.42) ;
\draw   (119.33,321.75) -- (134,329.08) -- (119.33,336.42) ;
\draw   (490.33,322.75) -- (505,330.08) -- (490.33,337.42) ;
\draw   (493.33,269.75) -- (508,277.08) -- (493.33,284.42) ;
\draw   (492.33,218.75) -- (507,226.08) -- (492.33,233.42) ;
\draw   (494.33,168.75) -- (509,176.08) -- (494.33,183.42) ;
\draw   (493.33,122.75) -- (508,130.08) -- (493.33,137.42) ;
\draw   (495.33,79.75) -- (510,87.08) -- (495.33,94.42) ;
\draw [color={rgb, 255:red, 255; green, 0; blue, 0 }  ,draw opacity=1 ][line width=2.25]    (308.8,86.22) -- (454.8,86.22) ;

\draw [color={rgb, 255:red, 255; green, 0; blue, 0 }  ,draw opacity=1 ][line width=2.25]    (188.8,131.22) -- (421.8,132.22) ;

\draw    (551.93,130.55) -- (421.8,132.22) ;

\draw    (87,130.55) -- (188.8,131.22) ;

\draw [color={rgb, 255:red, 255; green, 0; blue, 0 }  ,draw opacity=1 ][line width=2.25]    (189.8,177.22) -- (424.8,176.08) ;

\draw    (87,176.55) -- (189.8,177.22) ;

\draw    (551.93,176.55) -- (424.8,176.08) ;

\draw    (87,86.42) -- (308.8,86.22) ;

\draw    (454.8,86.22) -- (551.93,86.42) ;

\draw    (87,225.55) -- (189.8,225.22) ;

\draw [color={rgb, 255:red, 255; green, 0; blue, 0 }  ,draw opacity=1 ][line width=2.25]    (189.8,225.22) -- (425.8,225.22) ;

\draw    (425.8,225.22) -- (551.93,225.55) ;

\draw    (87,276.55) -- (189.8,277.22) ;

\draw [color={rgb, 255:red, 255; green, 0; blue, 0 }  ,draw opacity=1 ][line width=2.25]    (426.8,276.22) -- (189.8,277.22) ;

\draw    (426.8,276.22) -- (551.93,276.55) ;

\draw    (189.8,177.22) -- (426.8,276.22) ;

\draw (312,73.42) node   {$e'_{1}$};
\draw (458,72.42) node   {$e'_{4}$};
\draw (172,117.42) node   {$e'_{m}$};
\draw (434,119.42) node   {$e'_{3}$};
\draw (437,166.42) node   {$e'_{2}$};
\draw (172,165.42) node   {$e'_{m-1}$};
\draw (173,212.42) node   {$e_{n}$};
\draw (171,265.42) node   {$e_{n-1}$};
\draw (170,316.42) node   {$e_{n-2}$};
\draw (439,214.42) node   {$e_{3}$};
\draw (439,265.42) node   {$e_{2}$};
\draw (448,316.42) node   {$e_{4}$};
\draw (311,338.42) node   {$e_{1}$};

\end{tikzpicture}

\caption{$D$: edges marked thick denote two directed edges between corresponding vertices}  
\end{figure}
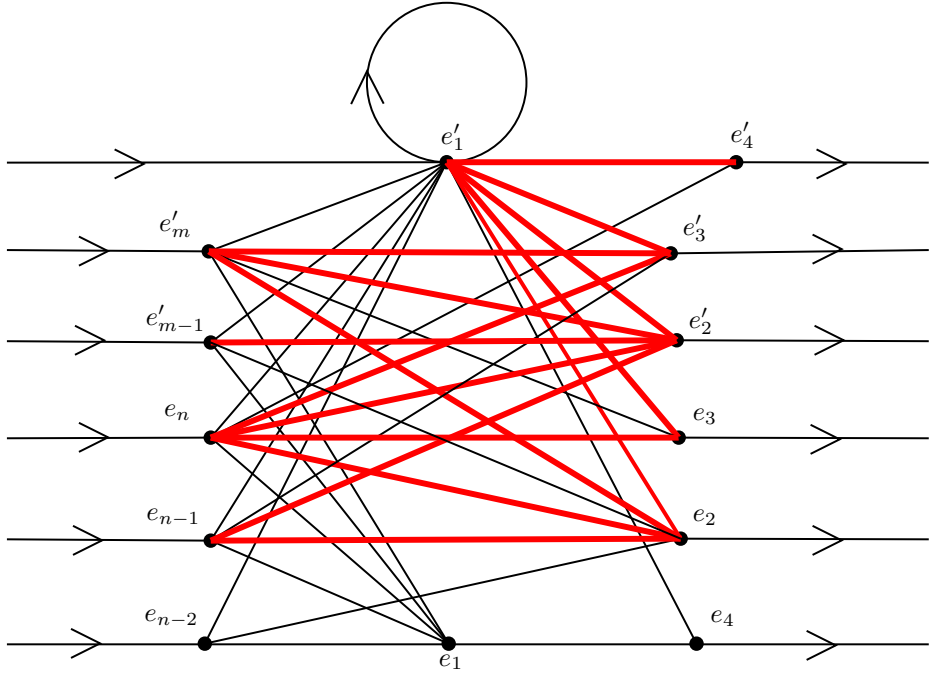

Assume $n=m$, we get the directed graph $\Gamma$ associated to $f$ (or $g$) and $T$ (with only the contributing edges) as shown in Figure 8. The graph is made of 6 big ``loops" going clockwise, together with a subgraph $D$. The subgraph $D$ is given by the relations determined by $g$ above, as shown in Figure 9, containing one loop at $e\textprime _1$. For simplicity, the graph of $D$ in Figure 9 omits the arrows in between. All edges with omitted arrows implicitly point from left to right. The edges marked thick mean there are two edges connecting those vertices. Thus, a path with given length passing through $D$ once will either
\begin{itemize}
\item directly go from left to right with length 1.
\item go from left to $e\textprime_1$, then wrap around the loop at $e\textprime_1$ some number of times, then go to the right.
\item pass $e_1$ and go to $e_4$. 
\end{itemize}
Given two vertices, the number of paths of length $\frac{n}{13}$ between them which passes through $D$ is therefore at most 2.

Now we let $\Sigma_0$ surround $V_{\lfloor \frac{n}{2} \rfloor-1}, V_{\lfloor \frac{n}{2} \rfloor},V_{\lfloor \frac{n}{2} \rfloor+1}$, fix $h_k$ and consider a graph map $g_k\simeq h_kf$ and its matrix $T_k$. Note that $h_k$ is supported in a neighborhood of $\Sigma_0$. Let $a_j, a_{j+1}, a_{j+2}$ denote the three loops wrapping around the three punctures in $\Sigma_0$. If we remove all the non-contributing edges, after homotopy, $h_k$ sends $e_j,e_{j+1}, e_{j+2}$ to a combination of $e_j,e_{j+1}, e_{j+2}$ without acting on other edges. Thus $g_k\simeq h_kf$ sends $e_{j-3},e_{j-2}, e_{j-1}$ to a combination of $e_j,e_{j+1}, e_{j+2}$ and acts on the rest of the edges as $g\simeq f$ does. 

Then we get the directed graph $\Gamma_k$ associated to $T_k$ and $g_k$ as shown in Figure 10. The graph $\Gamma_k$ is the same as $\Gamma$ away from $e_{j-3},e_{j-2},  e_{j-1},e_j,e_{j+1}, e_{j+2}$, and has a subgraph $D_k$ given by $h_k$. The subgraph $D_k$ is a bipartite graph with 3 vertices in each set, $\{e_j,e_{j+1}, e_{j+2}\}$ and $\{e_{j-3},e_{j-2}, e_{j-1\}}$. All edges of $D_k$ point from right to left, from $\{e_{j-3},e_{j-2}, e_{j-1}\}$ to $\{e_j,e_{j+1}, e_{j+2}\}$. The number of edges between any two vertices in different sets is bounded above by some $E_k>0$ depending on $h_k$. See Figure 11.

When $n=m$ is big enough, any path of length $\frac{n}{13}$ can't intersect $D$ and $D_k$ simultaneously. Thus given any two vertices, the number of paths of length $\frac{n}{13}$ between those vertices is bounded above by $N_k=\max\{2, E_k\}$. The number of paths of length $\frac{n}{13}$ emanating from a given vertex is thus at most $2nN_k$. Then for $\lambda_0$, the leading eigenvalue of $T_k$, by Proposition \ref{pf}, we have
\[\log \lambda_0 \leq \frac{\log 2nN_k}{\frac{n}{13}}.\]

When $n>N_k$ is large enough, we have
\[
\log \lambda_0 \leq \frac{\log 2nN_k}{\frac{n}{13}} < \frac{2\log (2n+2)}{\frac{2n}{26}} < \frac{2\log (2n+2)}{\frac{2n+2}{27}}=54\frac{\log (2n+2)}{2n+2}.
\]
The result follows since $\lambda(h_kf)\leq \lambda_0$.
\end{proof}
\begin{figure}[H]\centering

\tikzset{every picture/.style={line width=0.75pt}} 

\begin{tikzpicture}[x=0.75pt,y=0.75pt,yscale=-1,xscale=1]

\draw   (7,151.45) .. controls (7,82.01) and (105.56,25.73) .. (227.13,25.73) .. controls (348.71,25.73) and (447.27,82.01) .. (447.27,151.45) .. controls (447.27,220.88) and (348.71,277.17) .. (227.13,277.17) .. controls (105.56,277.17) and (7,220.88) .. (7,151.45) -- cycle ;
\draw   (31.29,151.45) .. controls (31.29,89.68) and (118.97,39.6) .. (227.13,39.6) .. controls (335.3,39.6) and (422.98,89.68) .. (422.98,151.45) .. controls (422.98,213.22) and (335.3,263.3) .. (227.13,263.3) .. controls (118.97,263.3) and (31.29,213.22) .. (31.29,151.45) -- cycle ;
\draw   (54.74,151.45) .. controls (54.74,97.07) and (131.92,52.99) .. (227.13,52.99) .. controls (322.34,52.99) and (399.53,97.07) .. (399.53,151.45) .. controls (399.53,205.82) and (322.34,249.9) .. (227.13,249.9) .. controls (131.92,249.9) and (54.74,205.82) .. (54.74,151.45) -- cycle ;
\draw   (80.28,151.45) .. controls (80.28,105.13) and (146.03,67.58) .. (227.13,67.58) .. controls (308.23,67.58) and (373.98,105.13) .. (373.98,151.45) .. controls (373.98,197.77) and (308.23,235.31) .. (227.13,235.31) .. controls (146.03,235.31) and (80.28,197.77) .. (80.28,151.45) -- cycle ;
\draw   (106.42,151.45) .. controls (106.42,113.37) and (160.47,82.51) .. (227.13,82.51) .. controls (293.8,82.51) and (347.85,113.37) .. (347.85,151.45) .. controls (347.85,189.52) and (293.8,220.39) .. (227.13,220.39) .. controls (160.47,220.39) and (106.42,189.52) .. (106.42,151.45) -- cycle ;
\draw   (133.69,151.45) .. controls (133.69,121.98) and (175.53,98.08) .. (227.13,98.08) .. controls (278.74,98.08) and (320.57,121.98) .. (320.57,151.45) .. controls (320.57,180.92) and (278.74,204.81) .. (227.13,204.81) .. controls (175.53,204.81) and (133.69,180.92) .. (133.69,151.45) -- cycle ;
\draw  [fill={rgb, 255:red, 255; green, 255; blue, 255 }  ,fill opacity=1 ] (194.83,8.83) -- (255.6,8.83) -- (255.6,115.17) -- (194.83,115.17) -- cycle ;
\draw  [fill={rgb, 255:red, 255; green, 255; blue, 255 }  ,fill opacity=1 ] (195.61,194.82) -- (255.31,194.82) -- (255.31,241.25) -- (195.61,241.25) -- cycle ;
\draw  [color={rgb, 255:red, 0; green, 0; blue, 0 }  ][line width=2.25] [line join = round][line cap = round] (128.39,38.65) .. controls (128.39,38.65) and (128.39,38.65) .. (128.39,38.65) ;
\draw  [color={rgb, 255:red, 0; green, 0; blue, 0 }  ][line width=2.25] [line join = round][line cap = round] (161.61,30.98) .. controls (161.61,30.98) and (161.61,30.98) .. (161.61,30.98) ;
\draw  [color={rgb, 255:red, 0; green, 0; blue, 0 }  ][line width=2.25] [line join = round][line cap = round] (163.31,45.46) .. controls (163.31,45.46) and (163.31,45.46) .. (163.31,45.46) ;
\draw  [color={rgb, 255:red, 0; green, 0; blue, 0 }  ][line width=2.25] [line join = round][line cap = round] (130.09,53.98) .. controls (130.09,53.98) and (130.09,53.98) .. (130.09,53.98) ;
\draw  [color={rgb, 255:red, 0; green, 0; blue, 0 }  ][line width=2.25] [line join = round][line cap = round] (165.02,59.09) .. controls (165.02,59.09) and (165.02,59.09) .. (165.02,59.09) ;
\draw  [color={rgb, 255:red, 0; green, 0; blue, 0 }  ][line width=2.25] [line join = round][line cap = round] (133.5,68.46) .. controls (133.5,68.46) and (133.5,68.46) .. (133.5,68.46) ;
\draw  [color={rgb, 255:red, 0; green, 0; blue, 0 }  ][line width=2.25] [line join = round][line cap = round] (165.87,75.28) .. controls (165.87,75.28) and (165.87,75.28) .. (165.87,75.28) ;
\draw  [color={rgb, 255:red, 0; green, 0; blue, 0 }  ][line width=2.25] [line join = round][line cap = round] (136.91,85.5) .. controls (136.91,85.5) and (136.91,85.5) .. (136.91,85.5) ;
\draw  [color={rgb, 255:red, 0; green, 0; blue, 0 }  ][line width=2.25] [line join = round][line cap = round] (166.72,91.46) .. controls (166.72,91.46) and (166.72,91.46) .. (166.72,91.46) ;
\draw  [color={rgb, 255:red, 0; green, 0; blue, 0 }  ][line width=2.25] [line join = round][line cap = round] (142.02,102.54) .. controls (142.02,102.54) and (142.02,102.54) .. (142.02,102.54) ;
\draw  [color={rgb, 255:red, 0; green, 0; blue, 0 }  ][line width=2.25] [line join = round][line cap = round] (169.28,109.35) .. controls (169.28,109.35) and (169.28,109.35) .. (169.28,109.35) ;
\draw  [color={rgb, 255:red, 0; green, 0; blue, 0 }  ][line width=2.25] [line join = round][line cap = round] (145.43,125.54) .. controls (145.43,125.54) and (145.43,125.54) .. (145.43,125.54) ;
\draw  [color={rgb, 255:red, 0; green, 0; blue, 0 }  ][line width=2.25] [line join = round][line cap = round] (325.02,38.65) .. controls (325.02,38.65) and (325.02,38.65) .. (325.02,38.65) ;
\draw  [color={rgb, 255:red, 0; green, 0; blue, 0 }  ][line width=2.25] [line join = round][line cap = round] (291.8,30.98) .. controls (291.8,30.98) and (291.8,30.98) .. (291.8,30.98) ;
\draw  [color={rgb, 255:red, 0; green, 0; blue, 0 }  ][line width=2.25] [line join = round][line cap = round] (290.1,45.46) .. controls (290.1,45.46) and (290.1,45.46) .. (290.1,45.46) ;
\draw  [color={rgb, 255:red, 0; green, 0; blue, 0 }  ][line width=2.25] [line join = round][line cap = round] (323.32,53.98) .. controls (323.32,53.98) and (323.32,53.98) .. (323.32,53.98) ;
\draw  [color={rgb, 255:red, 0; green, 0; blue, 0 }  ][line width=2.25] [line join = round][line cap = round] (288.4,59.09) .. controls (288.4,59.09) and (288.4,59.09) .. (288.4,59.09) ;
\draw  [color={rgb, 255:red, 0; green, 0; blue, 0 }  ][line width=2.25] [line join = round][line cap = round] (319.91,68.46) .. controls (319.91,68.46) and (319.91,68.46) .. (319.91,68.46) ;
\draw  [color={rgb, 255:red, 0; green, 0; blue, 0 }  ][line width=2.25] [line join = round][line cap = round] (287.54,75.28) .. controls (287.54,75.28) and (287.54,75.28) .. (287.54,75.28) ;
\draw  [color={rgb, 255:red, 0; green, 0; blue, 0 }  ][line width=2.25] [line join = round][line cap = round] (316.51,84.65) .. controls (316.51,84.65) and (316.51,84.65) .. (316.51,84.65) ;
\draw  [color={rgb, 255:red, 0; green, 0; blue, 0 }  ][line width=2.25] [line join = round][line cap = round] (286.69,91.46) .. controls (286.69,91.46) and (286.69,91.46) .. (286.69,91.46) ;
\draw  [color={rgb, 255:red, 0; green, 0; blue, 0 }  ][line width=2.25] [line join = round][line cap = round] (311.4,102.54) .. controls (311.4,102.54) and (311.4,102.54) .. (311.4,102.54) ;
\draw  [color={rgb, 255:red, 0; green, 0; blue, 0 }  ][line width=2.25] [line join = round][line cap = round] (284.14,109.35) .. controls (284.14,109.35) and (284.14,109.35) .. (284.14,109.35) ;
\draw  [color={rgb, 255:red, 0; green, 0; blue, 0 }  ][line width=2.25] [line join = round][line cap = round] (308.84,125.54) .. controls (308.84,125.54) and (308.84,125.54) .. (308.84,125.54) ;
\draw   (174.54,26.04) -- (182.39,27.99) -- (176.12,33.11) ;
\draw   (139.61,32) -- (147.47,33.95) -- (141.19,39.07) ;
\draw   (174.54,40.52) -- (182.39,42.47) -- (176.12,47.59) ;
\draw   (139.61,47.34) -- (147.47,49.28) -- (141.19,54.4) ;
\draw   (174.54,54.15) -- (182.39,56.1) -- (176.12,61.22) ;
\draw   (143.02,60.97) -- (150.88,62.91) -- (144.6,68.03) ;
\draw   (174.54,69.48) -- (182.39,71.43) -- (176.12,76.55) ;
\draw   (145.15,77.99) -- (153.23,78.47) -- (147.99,84.64) ;
\draw   (175.39,85.67) -- (183.25,87.62) -- (176.97,92.74) ;
\draw   (175.13,103.16) -- (183.15,104.3) -- (177.43,110.03) ;
\draw   (147.65,95.15) -- (155.74,95.42) -- (150.67,101.73) ;
\draw   (151.68,115.6) -- (159.69,114.47) -- (155.79,121.56) ;
\draw   (268.07,25.18) -- (275.33,28.76) -- (268.12,32.42) ;
\draw   (305.31,30.62) -- (311.92,35.29) -- (304.22,37.78) ;
\draw   (269.45,38.35) -- (276.15,42.89) -- (268.5,45.52) ;
\draw   (267.75,51.98) -- (274.45,56.52) -- (266.8,59.15) ;
\draw   (266.9,67.31) -- (273.6,71.85) -- (265.94,74.49) ;
\draw   (269.45,83.49) -- (276.15,88.03) -- (268.5,90.67) ;
\draw   (268.29,100.17) -- (274.36,105.52) -- (266.43,107.17) ;
\draw   (304.9,44.81) -- (311,50.13) -- (303.08,51.82) ;
\draw   (301.52,58.42) -- (307.58,63.79) -- (299.65,65.42) ;
\draw   (299.1,74.53) -- (305,80.08) -- (297.02,81.47) ;
\draw   (297.86,91.35) -- (303.17,97.46) -- (295.1,98.04) ;
\draw   (294.79,110.82) -- (299.66,117.28) -- (291.56,117.3) ;

\draw (225.93,57.74) node   {$D$};
\draw (227.13,218.15) node   {$D_{k}$};
\draw (162.04,23.31) node [scale=0.7]  {$e'_{m-2}$};
\draw (124.56,30.98) node [scale=0.7]  {$e'_{m-5}$};
\draw (162.04,38.65) node [scale=0.7]  {$e'_{m-3}$};
\draw (126.26,47.17) node [scale=0.7]  {$e'_{m-6}$};
\draw (129.67,61.65) node [scale=0.7]  {$e'_{m-7}$};
\draw (161.19,53.13) node [scale=0.7]  {$e'_{m-4}$};
\draw (163.74,68.46) node [scale=0.7]  {$e_{n-3}$};
\draw (133.93,77.83) node [scale=0.7]  {$e_{n-6}$};
\draw (138.19,94.02) node [scale=0.7]  {$e_{n-7}$};
\draw (164.59,84.65) node [scale=0.7]  {$e_{n-4}$};
\draw (168,100.83) node [scale=0.7]  {$e_{n-5}$};
\draw (139.89,116.17) node [scale=0.7]  {$e_{n-8}$};
\draw (93.89,36.23) node [scale=0.9,rotate=-335.11]  {$\dotsc $};
\draw (99.85,53.27) node [scale=0.9,rotate=-335.11]  {$\dotsc $};
\draw (104.96,67.75) node [scale=0.9,rotate=-335.11]  {$\dotsc $};
\draw (110.07,83.09) node [scale=0.9,rotate=-335.11]  {$\dotsc $};
\draw (116.04,102.68) node [scale=0.9,rotate=-335.11]  {$\dotsc $};
\draw (125.41,125.68) node [scale=0.9,rotate=-335.11]  {$\dotsc $};
\draw (291.52,23.31) node [scale=0.7]  {$e'_{7}$};
\draw (325.59,30.98) node [scale=0.7]  {$e'_{10}$};
\draw (291.52,38.65) node [scale=0.7]  {$e'_{6}$};
\draw (324.74,48.02) node [scale=0.7]  {$e'_{9}$};
\draw (288.96,52.28) node [scale=0.7]  {$e'_{5}$};
\draw (322.19,61.65) node [scale=0.7]  {$e'_{8}$};
\draw (288.11,68.46) node [scale=0.7]  {$e_{6}$};
\draw (317.93,78.69) node [scale=0.7]  {$e_{9}$};
\draw (287.26,84.65) node [scale=0.7]  {$e_{5}$};
\draw (312.81,95.72) node [scale=0.7]  {$e_{8}$};
\draw (284.7,101.69) node [scale=0.7]  {$e_{7}$};
\draw (311.96,117.87) node [scale=0.7]  {$e_{10}$};
\draw (346.89,34.53) node [scale=0.9,rotate=-17.84]  {$\dotsc $};
\draw (343.48,51.57) node [scale=0.9,rotate=-17.84]  {$\dotsc $};
\draw (340.93,66.05) node [scale=0.9,rotate=-17.84]  {$\dotsc $};
\draw (334.96,83.94) node [scale=0.9,rotate=-17.84]  {$\dotsc $};
\draw (331.56,100.12) node [scale=0.9,rotate=-17.84]  {$\dotsc $};
\draw (329,123.12) node [scale=0.9,rotate=-17.84]  {$\dotsc $};

\end{tikzpicture}

\caption{The directed graph $\Gamma_k$ associated to $h_kf$.}  
\end{figure}
\begin{figure}[H]\centering

\tikzset{every picture/.style={line width=0.75pt}} 

\begin{tikzpicture}[x=0.75pt,y=0.75pt,yscale=-1,xscale=1]

\draw    (104.83,239.98) -- (520.93,239.98) ;

\draw  [fill={rgb, 255:red, 0; green, 0; blue, 0 }  ,fill opacity=1 ] (191.35,239.68) .. controls (191.35,238.13) and (192.6,236.87) .. (194.15,236.87) .. controls (195.7,236.87) and (196.96,238.13) .. (196.96,239.68) .. controls (196.96,241.23) and (195.7,242.48) .. (194.15,242.48) .. controls (192.6,242.48) and (191.35,241.23) .. (191.35,239.68) -- cycle ;
\draw  [fill={rgb, 255:red, 0; green, 0; blue, 0 }  ,fill opacity=1 ] (413.3,239.68) .. controls (413.3,238.13) and (414.55,236.87) .. (416.1,236.87) .. controls (417.65,236.87) and (418.91,238.13) .. (418.91,239.68) .. controls (418.91,241.23) and (417.65,242.48) .. (416.1,242.48) .. controls (414.55,242.48) and (413.3,241.23) .. (413.3,239.68) -- cycle ;
\draw  [fill={rgb, 255:red, 0; green, 0; blue, 0 }  ,fill opacity=1 ] (194.03,146.6) .. controls (194.03,145.05) and (195.29,143.8) .. (196.84,143.8) .. controls (198.38,143.8) and (199.64,145.05) .. (199.64,146.6) .. controls (199.64,148.15) and (198.38,149.41) .. (196.84,149.41) .. controls (195.29,149.41) and (194.03,148.15) .. (194.03,146.6) -- cycle ;
\draw  [fill={rgb, 255:red, 0; green, 0; blue, 0 }  ,fill opacity=1 ] (194.03,193.14) .. controls (194.03,191.59) and (195.29,190.34) .. (196.84,190.34) .. controls (198.38,190.34) and (199.64,191.59) .. (199.64,193.14) .. controls (199.64,194.69) and (198.38,195.94) .. (196.84,195.94) .. controls (195.29,195.94) and (194.03,194.69) .. (194.03,193.14) -- cycle ;
\draw  [fill={rgb, 255:red, 0; green, 0; blue, 0 }  ,fill opacity=1 ] (405.24,146.6) .. controls (405.24,145.05) and (406.5,143.8) .. (408.05,143.8) .. controls (409.6,143.8) and (410.85,145.05) .. (410.85,146.6) .. controls (410.85,148.15) and (409.6,149.41) .. (408.05,149.41) .. controls (406.5,149.41) and (405.24,148.15) .. (405.24,146.6) -- cycle ;
\draw  [fill={rgb, 255:red, 0; green, 0; blue, 0 }  ,fill opacity=1 ] (406.14,192.25) .. controls (406.14,190.7) and (407.39,189.44) .. (408.94,189.44) .. controls (410.49,189.44) and (411.75,190.7) .. (411.75,192.25) .. controls (411.75,193.79) and (410.49,195.05) .. (408.94,195.05) .. controls (407.39,195.05) and (406.14,193.79) .. (406.14,192.25) -- cycle ;
\draw [color={rgb, 255:red, 0; green, 0; blue, 0 }  ,draw opacity=1 ]   (196.84,146.6) -- (408.94,192.25) ;

\draw    (194.15,239.68) -- (408.94,192.25) ;

\draw   (150.75,152.76) -- (137.36,146.76) -- (150.19,139.65) ;
\draw   (147.9,199.48) -- (134.67,193.13) -- (147.68,186.35) ;
\draw   (146.95,246.07) -- (133.77,239.62) -- (146.84,232.94) ;
\draw   (478.78,247.17) -- (465.81,240.3) -- (479.08,234.04) ;
\draw   (479.88,198.63) -- (466.7,192.19) -- (479.76,185.5) ;
\draw   (480.67,153.09) -- (467.59,146.43) -- (480.77,139.97) ;
\draw    (104.83,146.9) -- (196.84,146.6) ;

\draw [color={rgb, 255:red, 0; green, 0; blue, 0 }  ,draw opacity=1 ]   (196.84,146.6) -- (408.05,146.6) ;

\draw    (408.05,146.6) -- (520.93,146.9) ;

\draw    (104.83,192.54) -- (196.84,193.14) ;

\draw [color={rgb, 255:red, 0; green, 0; blue, 0 }  ,draw opacity=1 ]   (408.94,192.25) -- (196.84,193.14) ;

\draw    (408.94,192.25) -- (520.93,192.54) ;

\draw    (408.05,146.6) -- (196.84,193.14) ;

\draw    (408.05,146.6) -- (194.15,239.68) ;

\draw    (416.1,239.68) -- (196.84,193.14) ;

\draw    (196.84,146.6) -- (416.1,239.68) ;

\draw   (300.76,152.43) -- (289.57,146.66) -- (300.9,141.17) ;
\draw   (270.33,198.97) -- (259.14,193.2) -- (270.47,187.7) ;
\draw   (304.34,245.5) -- (293.15,239.74) -- (304.48,234.24) ;
\draw   (273.47,168.69) -- (263.76,160.68) -- (276,157.72) ;
\draw   (287.57,191.97) -- (280.42,181.6) -- (293,182.1) ;
\draw   (337.13,167.49) -- (324.72,165.37) -- (333.86,156.71) ;
\draw   (345.76,230.57) -- (336.29,222.27) -- (348.61,219.68) ;
\draw   (265.93,230.05) -- (253.86,226.47) -- (263.96,218.96) ;
\draw   (362.43,174.15) -- (349.9,172.9) -- (358.42,163.62) ;

\draw (181.8,135.15) node   {$e_{j+2}$};
\draw (180.01,182.58) node   {$e_{j+1}$};
\draw (179.12,228.22) node   {$e_{j}$};
\draw (419.86,136.94) node   {$e_{j-1}$};
\draw (419.86,182.58) node   {$e_{j-2}$};
\draw (427.92,228.22) node   {$e_{j-3}$};

\end{tikzpicture}

\caption{$D_k$: each directed edge in between represent $\leq E_k$ directed edge.}  
\end{figure}
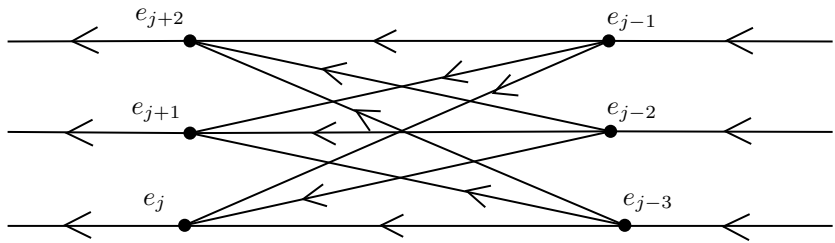

Now we finish Theorem \ref{main}. Part (1) is given by Lemma 2. Part (2) is given by Proposition 3. Part (3) is given by Lemma 4.

\bibliographystyle{amsalpha}
\bibliography{1}

\providecommand{\bysame}{\leavevmode\hbox to3em{\hrulefill}\thinspace}
\providecommand{\MR}{\relax\ifhmode\unskip\space\fi MR }
\providecommand{\MRhref}[2]{%
  \href{http://www.ams.org/mathscinet-getitem?mr=#1}{#2}
}
\providecommand{\href}[2]{#2}
\begin{thebibliography}{ALM16}

\bibitem[AD10]{bound4}
John~W. Aaber and Nathan Dunfield, \emph{Closed surface bundles of least
  volume}, Algebr. Geom. Topol. \textbf{10} (2010), no.~4, 2315--2342.
  \MR{2745673}

\bibitem[Ada85]{adams}
Colin~C. Adams, \emph{Thrice-punctured spheres in hyperbolic {$3$}-manifolds},
  Trans. Amer. Math. Soc. \textbf{287} (1985), no.~2, 645--656. \MR{768730}

\bibitem[Ago03]{agol}
Ian Agol, \emph{Small 3-manifolds of large genus}, Geom. Dedicata \textbf{102}
  (2003), 53--64. \MR{2026837}

\bibitem[Ago11]{agol2}
\bysame, \emph{Ideal triangulations of pseudo-{A}nosov mapping tori}, Topology
  and geometry in dimension three, Contemp. Math., vol. 560, Amer. Math. Soc.,
  Providence, RI, 2011, pp.~1--17. \MR{2866919}

\bibitem[ALM16]{bound6}
Ian Agol, Christopher~J. Leininger, and Dan Margalit, \emph{Pseudo-{A}nosov
  stretch factors and homology of mapping tori}, J. Lond. Math. Soc. (2)
  \textbf{93} (2016), no.~3, 664--682. \MR{3509958}

\bibitem[Bau92]{bound1}
Max Bauer, \emph{An upper bound for the least dilatation}, Trans. Amer. Math.
  Soc. \textbf{330} (1992), no.~1, 361--370. \MR{1094556}

\bibitem[BB16]{brock}
Jeffrey~F. Brock and Kenneth~W. Bromberg, \emph{Inflexibility,
  {W}eil-{P}eterson distance, and volumes of fibered 3-manifolds}, Math. Res.
  Lett. \textbf{23} (2016), no.~3, 649--674. \MR{3533189}

\bibitem[BH95]{bestvina}
M.~Bestvina and M.~Handel, \emph{Train-tracks for surface homeomorphisms},
  Topology \textbf{34} (1995), no.~1, 109--140. \MR{1308491}

\bibitem[FLM11]{flm}
Benson Farb, Christopher~J. Leininger, and Dan Margalit, \emph{Small dilatation
  pseudo-{A}nosov homeomorphisms and 3-manifolds}, Adv. Math. \textbf{228}
  (2011), no.~3, 1466--1502. \MR{2824561}

\bibitem[FLP12]{thurston}
Albert Fathi, Fran\c{c}ois Laudenbach, and Valentin Po\'{e}naru,
  \emph{Thurston's work on surfaces}, Mathematical Notes, vol.~48, Princeton
  University Press, Princeton, NJ, 2012, Translated from the 1979 French
  original by Djun M. Kim and Dan Margalit. \MR{3053012}

\bibitem[Fri79]{fried1}
David Fried, \emph{Fibrations over $s^1$ with pseudo-anosov monodromy}, Travaux
  de Thurston sur les surfaces - S\'eminaire Orsay, Ast\'erisque, no. 66-67,
  Soci\'et\'e math\'ematique de France, 1979, pp.~251--266 (en).

\bibitem[Fri82]{fried2}
\bysame, \emph{The geometry of cross sections to flows}, Topology \textbf{21}
  (1982), no.~4, 353--371. \MR{670741}

\bibitem[Gan59]{gantmacher}
F.~R. Gantmacher, \emph{The theory of matrices. {V}ols. 1, 2}, Translated by K.
  A. Hirsch, Chelsea Publishing Co., New York, 1959. \MR{0107649}

\bibitem[HK06a]{bound2}
Eriko Hironaka and Eiko Kin, \emph{A family of pseudo-{A}nosov braids with
  small dilatation}, Algebr. Geom. Topol. \textbf{6} (2006), 699--738.
  \MR{2240913}

\bibitem[HK06b]{hironaka}
\bysame, \emph{A family of pseudo-{A}nosov braids with small dilatation},
  Algebr. Geom. Topol. \textbf{6} (2006), 699--738. \MR{2240913}

\bibitem[KM18]{kojima}
Sadayoshi Kojima and Greg McShane, \emph{Normalized entropy versus volume for
  pseudo-{A}nosovs}, Geom. Topol. \textbf{22} (2018), no.~4, 2403--2426.
  \MR{3784525}

\bibitem[KT16]{bound5}
Eiko Kin and Mitsuhiko Takasawa, \emph{The boundary of a fibered face of the
  magic 3-manifold and the asymptotic behavior of minimal pseudo-{A}nosov
  dilatations}, Hiroshima Math. J. \textbf{46} (2016), no.~3, 271--287.
  \MR{3614298}

\bibitem[LM86]{long}
D.~D. Long and H.~R. Morton, \emph{Hyperbolic {$3$}-manifolds and surface
  automorphisms}, Topology \textbf{25} (1986), no.~4, 575--583. \MR{862441}

\bibitem[Min06]{bound3}
Hiroyuki Minakawa, \emph{Examples of pseudo-{A}nosov homeomorphisms with small
  dilatations}, J. Math. Sci. Univ. Tokyo \textbf{13} (2006), no.~2, 95--111.
  \MR{2277516}

\bibitem[Mor84]{morgan}
John~W. Morgan, \emph{On {T}hurston's uniformization theorem for
  three-dimensional manifolds}, The {S}mith conjecture ({N}ew {Y}ork, 1979),
  Pure Appl. Math., vol. 112, Academic Press, Orlando, FL, 1984, pp.~37--125.
  \MR{758464}

\bibitem[NZ85]{dehnsurgery}
Walter~D. Neumann and Don Zagier, \emph{Volumes of hyperbolic three-manifolds},
  Topology \textbf{24} (1985), no.~3, 307--332. \MR{815482}

\bibitem[Ota96]{otal}
Jean-Pierre Otal, \emph{Le th\'{e}or\`eme d'hyperbolisation pour les
  vari\'{e}t\'{e}s fibr\'{e}es de dimension 3}, Ast\'{e}risque (1996), no.~235,
  x+159. \MR{1402300}

\bibitem[Pen91]{penner}
R.~C. Penner, \emph{Bounds on least dilatations}, Proc. Amer. Math. Soc.
  \textbf{113} (1991), no.~2, 443--450. \MR{1068128}

\bibitem[Sta78]{stallings}
John~R. Stallings, \emph{Constructions of fibred knots and links}, Algebraic
  and geometric topology ({P}roc. {S}ympos. {P}ure {M}ath., {S}tanford {U}niv.,
  {S}tanford, {C}alif., 1976), {P}art 2, Proc. Sympos. Pure Math., XXXII, Amer.
  Math. Soc., Providence, R.I., 1978, pp.~55--60. \MR{520522}

\bibitem[Thu78]{thurstonnote}
William~P. Thurston, \emph{Geometry and topology of $3$-manifolds, lecture
  notes}, Princeton University (1978).

\bibitem[Thu82]{thurston1}
\bysame, \emph{Three-dimensional manifolds, {K}leinian groups and hyperbolic
  geometry}, Bull. Amer. Math. Soc. (N.S.) \textbf{6} (1982), no.~3, 357--381.
  \MR{648524}

\bibitem[Thu86]{thurstonnorm}
\bysame, \emph{A norm for the homology of {$3$}-manifolds}, Mem. Amer. Math.
  Soc. \textbf{59} (1986), no.~339, i--vi and 99--130. \MR{823443}

\bibitem[Tsa09]{tsai}
Chia-Yen Tsai, \emph{The asymptotic behavior of least pseudo-{A}nosov
  dilatations}, Geom. Topol. \textbf{13} (2009), no.~4, 2253--2278.
  \MR{2507119}

\bibitem[Val12]{valdivia}
Aaron~D. Valdivia, \emph{Sequences of pseudo-{A}nosov mapping classes and their
  asymptotic behavior}, New York J. Math. \textbf{18} (2012), 609--620.
  \MR{2967106}

\bibitem[Yaz18a]{yazdi}
Mehdi Yazdi, \emph{Lower bound for dilatations}, J. Topol. \textbf{11} (2018),
  no.~3, 602--614. \MR{3830877}

\bibitem[Yaz18b]{yazdi2}
Mehdi Yazdi, \emph{Pseudo-anosov maps with small stretch factors on punctured
  surfaces}, 2018.

\end{thebibliography}

\end{document}